\documentclass[arxiv]{agn_article}

\usepackage[markexamples=true,markremarks=true]{agn_all}
\usepackage{caption}
\usepackage{multicol}
\usepackage{subcaption}

\newcommand{\sample}{42}

\renewcommand{\Rn}{\R^n}

\DeclareMathOperator{\arcosh}{arcosh}
\DeclareMathOperator{\adj}{adj}
\newcommand{\lmax}{\lambdamax}
\newcommand{\lmin}{\lambdamin}

\newcommand{\Wmp}{W_{\mathrm{magic}}^+}

\newcommand{\Wiso}{W_{\mathrm{iso}}}
\newcommand{\Wvol}{W_{\mathrm{vol}}}

\newcommand{\Mps}{\mathfrak{M}_+^*}
\newcommand{\Ms}{\mathfrak{M}^*}

\newcommand{\ddr}{\frac{\mathrm{d}}{\mathrm{d}r}}

\newcommand{\casesif}{\caseif}
\DeclareMathOperator{\curl}{curl}
\DeclareMathOperator{\CCurl}{CCurl_{2D}}
\DeclareMathOperator{\curld}{curl_{2D}}
\DeclareMathOperator{\Curld}{Curl_{2D}}

\newcommand{\bary}[1]{\overline{#1}}
\providecommand{\meas}[1]{\abs{#1}}
\providecommand{\sob}[2]{W^{#1,#2}}

\providecommand{\dnu}{\intd\nu}

\usepackage{pifont}% http://ctan.org/pkg/pifont

\providecommand{\citet}[2][]{\citeauthor{#2} \cite[#1]{#2}}

\defineaffiliation{tud}{
	Institute of Numerical Mathematics,
	Technical University of Dresden,
	Zellescher Weg 12--14,
	01069 Dresden, Germany
}

\defineauthor{sander}{Oliver~Sander}{tud}{oliver.sander@tu-dresden.de}

\defineaffiliation{eth}{%
	Mechanics and Materials Lab,
	Department of Mechanical and Process Engineering,
	ETH Zürich,
	Leonhardstr. 21,
	8092 Zürich, Switzerland
}

\defineauthor{kochmann}{Dennis~M.~Kochmann}{eth}{dmk@ethz.ch}

\defineaffiliation{delft}{%
	Department of Materials Science and Engineering,
	Delft University of Technology,
	Mekelweg 2, 2628CD Delft, Netherlands
}

\defineauthor{sid}{Siddhant~Kumar}{delft}{Sid.Kumar@tudelft.nl}

\tikzset{
    state/.style={
           rectangle,
           rounded corners,
           draw=black, very thick,
           minimum height=2em,
           inner sep=6pt,
           text centered,
           }
}
\begin{document}
\title{Numerical approaches for investigating quasiconvexity in the context of Morrey's conjecture}
\date{\today}
\knownauthors[voss_d]{voss_d,martin,sander,sid,kochmann,neff}
\maketitle
\vspace{-1.5em}
\begin{abstract}
\noindent%
Deciding whether a given function is quasiconvex is generally a difficult task. Here, we discuss a number of numerical approaches that can be used in the search for a counterexample to the quasiconvexity of a given function $W$. We will demonstrate these methods using the planar isotropic rank-one convex function
\[
	W_{\rm magic}^+(F)=\frac{\lambdamax}{\lambdamin}-\log\frac{\lambdamax}{\lambdamin}+\log\det F=\frac{\lambdamax}{\lambdamin}+2\.\log\lambdamin\,,
\]
	where $\lambdamax\geq\lambdamin$ are the singular values of $F$, as our main example. In a previous contribution, we have shown that quasiconvexity of this function would imply quasiconvexity for all rank-one convex isotropic planar energies $W\col\GLp(2)\to\R$ with an additive volumetric-isochoric split of the form
\[
	W(F)=W_{\rm iso}(F)+W_{\rm vol}(\det F)=\widetilde W_{\rm iso}\bigg(\frac{F}{\sqrt{\det F}}\bigg)+W_{\rm vol}(\det F)
\]
with a concave volumetric part. This example is therefore of particular interest with regard to Morrey's open question whether or not rank-one convexity implies quasiconvexity in the planar case.
\end{abstract}

{\textbf{Key words:} nonlinear elasticity, hyperelasticity, planar elasticity, rank-one convexity, quasiconvexity, ellipticity, isotropy, volumetric-isochoric split, finite elements, physics-informed neural networks,
\\[.65em]
\noindent\textbf{AMS 2010 subject classification:
	74B20, % nonlinear elasticity
	74A10, % theory of constitutive functions
	26B25  % generalized convexity of functions with multiple variables
}

{\parskip=-0.4mm \tableofcontents}

%
%
%
%
%%%%%%%%%%%%%%%%%%%%%%%%%%%%%%%%%%%%%%%%%%%%
\section{Introduction}
The question whether or not rank-one convexity implies quasiconvexity in the planar case is considered one of the major open problems in the calculus of variations \cite{pedregal2019rank,casadio1993algebraic,parry1995planar,kawohl1988quasiconvexity}. Morrey \cite{morrey1952quasi,morrey2009multiple} conjectured that this is not the case in general. Of course, in order to demonstrate that these convexity properties are indeed distinct, it is sufficient to identify a single rank-one convex function which is not quasiconvex, as was done by Sverak in the non-planar case \cite{sverak1992rank}. While numerous viable criteria are known for rank-one convexity, it remains highly difficult to explicitly show that a function is not quasiconvex. It is therefore common to apply numerics to the problem \cite{dacorogna1998some,pedregal1996some,bartels2004effective}.

In this article, we discuss different numerical approaches for demonstrating the non-quasiconvexity of a given function. We will primarily apply our methods to a single example, which has been the subject of a previous article \cite{agn_voss2021morrey}: The planar isotropic energy function $\Wmp$, which is defined via
\begin{equation}
\label{eq:wmpfirst}
	\Wmp(F)=\frac{\lambdamax}{\lambdamin}-\log\frac{\lambdamax}{\lambdamin}+\log\det F=\frac{\lambdamax}{\lambdamin}+2\.\log\lambdamin
\end{equation}
with $\lmax\geq\lmin>0$ as the ordered singular values of the deformation gradient $F=\nabla\varphi\in\GLp(2)$. While it has already been shown that this function is rank-one convex but nowhere strictly elliptic and not polyconvex \cite{agn_voss2021morrey}, it remains open whether or not $\Wmp$ is quasiconvex.

The energy candidate \eqref{eq:wmpfirst} is very compelling: It emerged from the investigation of planar isotropic elastic energies with a so-called additive volumetric-isochoric split, i.e.\ of energies that can be written as the sum of an isochoric part depending only on the product $\frac\lmax\lmin$ and a volumetric part depending only on $\lmax\lmin$. Details of this structure and its characteristics are discussed in Section \ref{sec:vol-iso}. It can be shown \cite{agn_voss2021morrey} that $\Wmp(F)$ is a limit case in the investigation of \enquote{least rank-one convex} candidates in the family of energy functions with an additive volumetric-isochoric split. In this specific setting, it was demonstrated that the question of quasiconvexity for $\Wmp(F)$ also determines whether or not any other rank-one convex function with an additive volumetric-isochoric split and a concave volumetric part is quasiconvex.

This work is primarily meant to serve as a guideline towards numerical investigations in the context of Morrey's conjecture. In particular, while the general methods described here are mostly applicable to a large class of energy functions,\footnote{We assume throughout that the energy function is sufficiently regular, cf.\ \cite{ball2000regularity,conti2005rank}.} we utilize a number of specific invariance properties exhibited by the particular example $\Wmp$, as discussed in Section \ref{sec:functionProperties}. Exploiting such properties can vastly improve the efficiency of numerical approaches and, when investigating other types of functions, it should be kept in mind that similar invariances could be identified and used to simplify the more general numerical algorithms for finding counterexamples to quasiconvexity.

The methods we describe reliably find such counterexamples for functions which are known to be non-elliptic and therefore non-quasiconvex. For the rank-one convex energy candidate $\Wmp$, a number of microstructures with the same energy level as the homogeneous deformation were (re-)discovered. However, we were unable to demonstrate the non-quasiconvexity of $\Wmp$ with any numerical approach. Although inconclusive, these numerical results certainly suggest that the function $\Wmp$ is in fact quasiconvex and therefore not suited for answering Morrey's conjecture.
%
%
%
%
%%%%%%%%%%%%%%%%%%%%%%%%%%%%%%%%%%%%%%%%%%%%%%%%%%%%%%%%%%%%
\subsection{Convexity properties of energy functions}\label{sec:convexity}

We start by recalling the classical definitions of generalized convexity properties \cite{Dacorogna08,silhavy1997mechanics,agn_schroder2010poly}.
\begin{definition}\label{def:quasiconvexity}
	The energy function $W\col\Rnn\to\R\cup\{+\infty\}$ is \emph{quasiconvex} if and only if
	\begin{equation}
		\int_\Omega W(F_0+\nabla\vartheta(x))\,\dx\geq\int_\Omega W(F_0)\,\dx=\abs\Omega\cdot W(F_0)\qquad\text{for all}\quad F_0\in\Rnn,\;\vartheta\in W_0^{1,\infty}(\Omega;\Rn)\label{eq:quasiconvexity}
	\end{equation}
	for any domain $\Omega\subset\R^n$ with Lebesgue measure $\abs{\Omega}$. The energy function is \emph{strictly quasiconvex} if the inequality in \eqref{eq:quasiconvexity} is strict for all $\vartheta\neq0$.
\end{definition}

While quasiconvexity, together with suitable growth conditions, is sufficient to ensure weak lower semi-continuity of the energy functional, it has the disadvantage of being notoriously difficult to prove or disprove directly. This led to the introduction of various sufficient and necessary conditions for quasiconvexity. 
\begin{definition}\label{def:polyconvexity}[\cite{ball1976convexity}]
	The energy function $W\col\Rnn\to\R\cup\{+\infty\}$ is \emph{polyconvex} if and only if there exists a convex function $P\col\R^{m(n)}\to\R\cup\{+\infty\}$ with
	\begin{equation}
		W(F)=P\bigl(F,\adj_2(F),\cdots,\adj_n(F)\bigr)\qquad\text{for all}\quad F\in\Rnn\,,\label{eq:polyconvex}
	\end{equation}
	with $\adj_i(F)\in\Rnn$ as the matrix of the determinants of all $i\times i$--minors of $F$ and $m(n) \colonequals \sum_{i=1}^n \binom{n}{i}^2$. The energy function is \emph{strictly polyconvex} if such a $P$ exists which is strictly convex.	
\end{definition}
For the planar case $n=2$, an energy $W\col\R^{2\times 2}\to\R\cup\{+\infty\}$ is polyconvex if and only if there exists a convex mapping $P\col\R^{2\times 2}\times\R\cong\R^5\to\R\cup\{+\infty\}$ with
\[
	W(F)=P(F,\det F)\qquad\text{for all}\quad F\in\R^{2\times 2}\,.
\]

Whereas polyconvexity provides a sufficient criterion for quasiconvexity, a necessary condition is given by the rank-one convexity of a function.
\begin{definition}\label{def:rank1convexity}
	The energy function $W\col\Rnn\to\R\cup\{+\infty\}$ is \emph{rank-one convex} if for $F_1,F_2\in\Rnn$,
	\begin{equation}
		W(t\.F_1+(1-t)\.F_2)\leq t\.W(F_1)+(1-t)\.W(F_2)\quad\text{for all }\;t\in(0,1)\quad\text{if }\;\rank(F_1-F_2)=1\,.\label{eq:rankOneConvexity}
	\end{equation}
	If the energy function is twice differentiable, rank-one convexity is equivalent to the \emph{Legendre-Hadamard ellipticity condition}
	\begin{equation}
		D^2W(F).(\xi\otimes\eta,\xi\otimes\eta)\geq 0 \qquad\text{for all}\quad F\in\Rnn,\;\xi,\eta\in\Rn,\label{LegendreHadamardEllipticity}
	\end{equation}
	which expresses the ellipticity of the Euler-Lagrange equation $\Div DW(\nabla\varphi)=0$ corresponding to the variational problem
	\begin{equation}
		I(\varphi)=\int_\Omega W(\nabla\varphi(x))\,\dx\to\min\,.\label{eq:minimizationProblem}
	\end{equation}
	The energy function is \emph{strictly rank-one convex} if inequality \eqref{eq:rankOneConvexity} is strict.
\end{definition}

Overall, for any $W\col\Rnn\to\R\cup\{+\infty\}$ we have the well known hierarchy \cite{ball1976convexity,ball1987does,Dacorogna08}
\begin{equation}
	\text{polyconvexity}\quad\implies\quad\text{quasiconvexity}\quad\implies\quad\text{rank-one convexity}\,,
\end{equation}
with none of the reverse implications holding in general for $n\geq3$. The remaining question whether rank-one convexity implies quasiconvexity for $n=2$ is known as Morrey's problem.

In continuum mechanics and related applications, energy functions are often more naturally defined on the group $\GLpn$ instead of $\Rnn$ since a deformation gradient $F$ with non-positive determinant would imply local self-intersection. For such functions we introduce:
\begin{definition}
	A function $W\col\GLpn\to\R$ is called quasi-/poly-/rank-one convex if the function
	\[
		\What\col\Rnn\to\R\cup\{+\infty\}\,,\qquad\What(F)=\begin{cases}
			W(F) &\casesif F\in\GLpn\,,\\
			+\infty &\casesif F\notin\GLpn\,,
		\end{cases}
	\]
is quasi-/poly-/rank-one convex.
\end{definition}
%
%
%
%
%
%
%%%%%%%%%%%%%%%%%%%%%%%%%%%%%%%%%%%%%%%%%%%%%%%%%%%%%%%%%%%%%%%%%
\subsection{Periodic boundary conditions}\label{sec:periodicBoundary}

The definition of quasiconvexity \eqref{eq:quasiconvexity} can be reformulated in terms of deformations with periodic boundary conditions. As will be demonstrated in Section \ref{sec:periodic}, this can be helpful to find periodic microstructures numerically. In the context of planar elasticity, periodic boundary conditions have the form
\begin{equation}
	\varphi(x)=F_0\.x+\vartheta_\#(x)
\end{equation}
with a homogeneous deformation $F_0\in\GLp(2)$ and a periodic superposition $\vartheta_\#\in W_{\textnormal{per}}^{1,\infty}(\Omega)$. The domain $\Omega$ has to be in such a geometrical shape that $\R^2$ can be covered by periodic replications of $\Omega$ and the superposition $\vartheta_\#$ must be $\Omega$-periodic, cf.\ Figure \ref{fig:periodicExample}. Seemingly, periodic boundary conditions are a more general concept than Dirichlet boundary conditions with test functions $\vartheta\subset W_0^{1,\infty}(\Omega;\Rn)$ because they also allow for changes on the boundary itself.\footnote{While a periodic superposition $\vartheta_\#\in W_{\textnormal{per}}^{1,\infty}([a,b])$ allows for more modifications of the homogeneous deformation $F_0\.x$ than a test function $\vartheta\in W_0^{1,\infty}([a,b])$ , the average rate of change of $\varphi(x)=F_0\.x+\vartheta_\#(x)$ remains constant, i.e.\
	\[
	\frac{1}{b-a}\int_a^bF_0+\vartheta_\#'(x)\,\dx=\frac{F_0\.b+\vartheta_\#(b)-F_0\.a-\vartheta_\#(a)}{b-a}=F_0\.\frac{b-a}{b-a}=F_0\,.
	\]}

\begin{figure}[h!]
	\centering
	\includegraphics[width=0.493\textwidth]{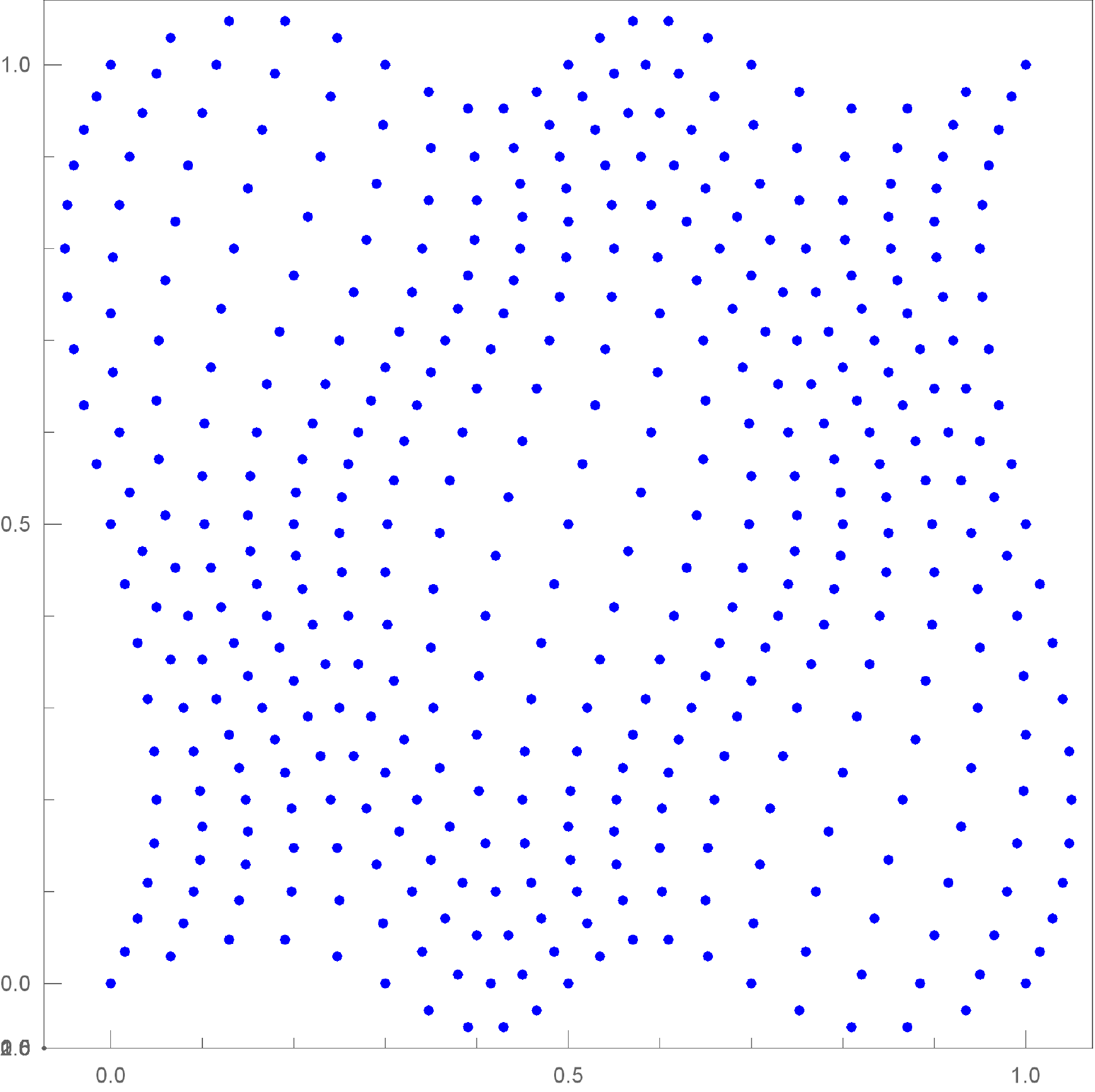}\hfill
	\includegraphics[width=0.49\textwidth]{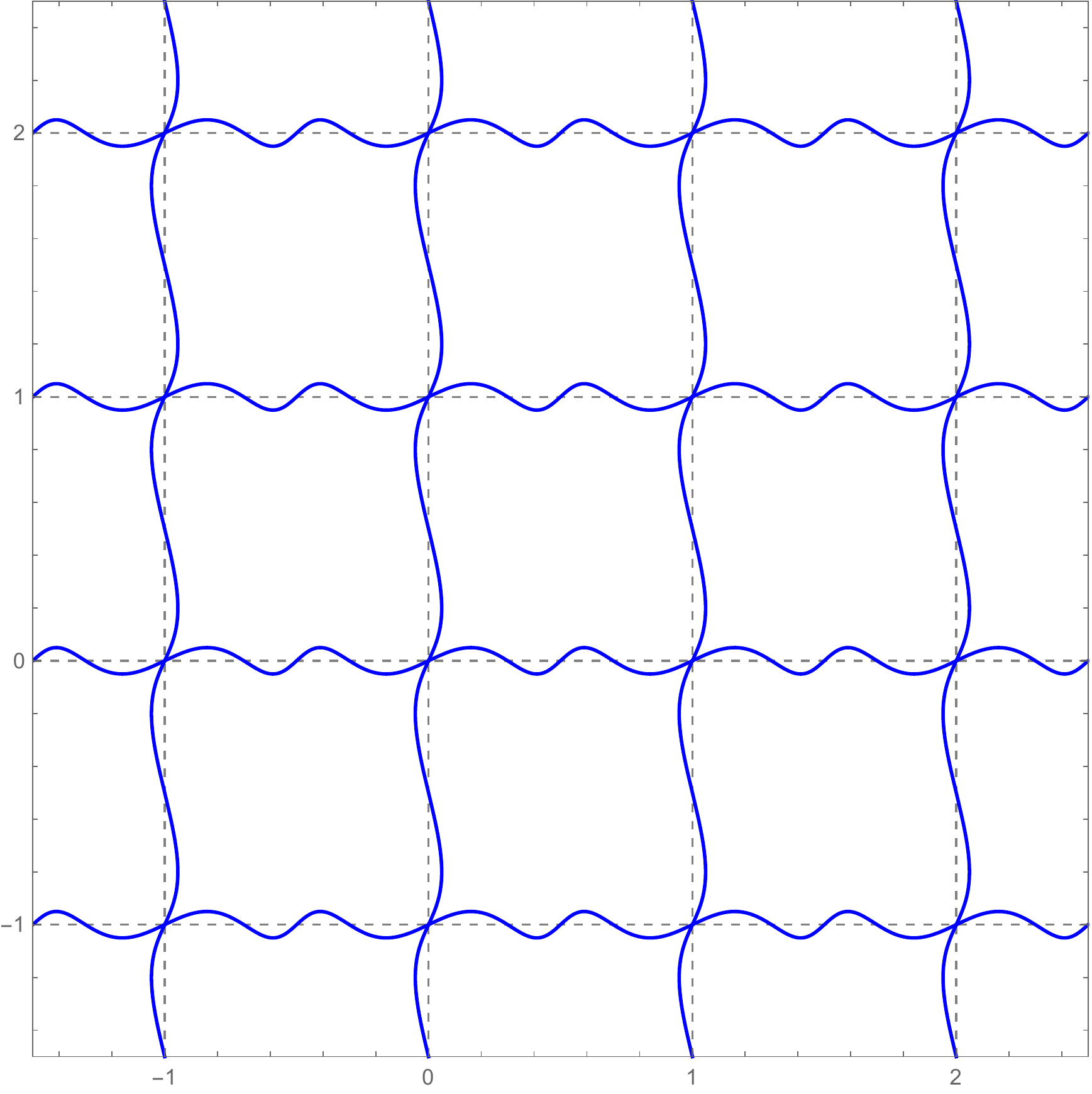}
	\caption{Example of a periodic deformation on the domain $\Omega=[0,1]^2$ which covers $\R^2$ by periodic replication.}\label{fig:periodicExample}
\end{figure}

However, for the notion of quasiconvexity, the two can be used interchangeably.
\begin{proposition}[{\cite[Proposition 5.13]{Dacorogna08}}]\label{def:quasiconvexityPeriodic}
	An energy function $W\col\Rnn\to\R$ is \emph{quasiconvex} if and only if
	\begin{equation}
		\int_\Omega W(F_0+\nabla\vartheta_\#(x))\,\dx\geq\int_\Omega W(F_0)\,\dx=\abs\Omega\cdot W(F_0)\qquad\text{for all}\quad F_0\in\Rnn,\;\vartheta_\#\in W_{\textnormal{per}}^{1,\infty}(\Omega)\label{eq:quasiconvexityPeriodic}
	\end{equation}
	for any domain $\Omega\subset\R^n$ with Lebesgue measure $\abs{\Omega}$ such tht $\R^2$ can be covered by periodic replications of $\Omega$. The energy function is \emph{strictly quasiconvex} if the inequality in \eqref{eq:quasiconvexityPeriodic} is strict for all $\vartheta\neq0$.
\end{proposition}
\begin{proof}
	Inequality \eqref{eq:quasiconvexityPeriodic} directly implies quasiconvexity, because $W_0^{1,\infty}(\Omega)\subset W_{\textnormal{per}}^{1,\infty}(\Omega)$. For the reverse direction see \cite[p.~173]{Dacorogna08}.
\end{proof}
In the context of planar elasticity, it is sufficient to consider $\Omega=[0,1]^2$, and without loss of generality we can assume $\vartheta_\#(0,0)=(0,0)$. The periodic replication of $\Omega$ to cover $\R^2$ implies that the values of $\vartheta_\#$ coincide on opposite edges, i.e.\ every point at the boundary belongs to two separate unit squares (four at the corners of the square). Thus we can write $\Omega$-periodicity as
\begin{align}
	\vartheta_\#(x,0)&=\vartheta_\#(x,1)\qquad\text{and}\qquad\nabla\vartheta_\#(x,0)=\nabla\vartheta_\#(x,1)\qquad\text{for all }x\in[0,1]\,,\notag\\
	\vartheta_\#(0,y)&=\vartheta_\#(1,y)\qquad\text{and}\qquad\nabla\vartheta_\#(0,y)=\nabla\vartheta_\#(1,y)\qquad\text{for all }y\in[0,1]\,;\label{eq:periodicBounderayConditions}
\end{align}
in particular,
\[
\vartheta_\#(0,0)=\vartheta_\#(1,0)=\vartheta_\#(0,1)=\vartheta_\#(1,1)=(0,0)\,.
\]
%
%
%
%
%
%
%%%%%%%%%%%%%%%%%%%%%%%%%%%%%%%%%%%%%%%%%%%%%%%%%%%%%%%%%%%%%%%%%
\subsection{Previous results related to Morrey's conjecture}

Morrey's conjecture has long been considered one of the most important open questions in the calculus of variations, and the remaining problem of the planar case has been the subject of extensive research.

The Dacorogna-Marcellini \cite{dacorogna1988counterexample} function $W\col\R^{2\times 2}\to\R$ with
\begin{align}
	W(F)=\norm{F}^2\left(\norm{F}^2-\gamma\.\det F\right)=(\lambdahat_1^2+\lambdahat_2^2)(\lambdahat_1^2+\lambdahat_2^2-\gamma\.\lambdahat_1\lambdahat_2)\,,\quad\gamma\in\R
\end{align}
is a homogeneous polynomial of degree four; here, $\lambdahat_i$ denote the signed singular values\footnote{Note that $\det F=\lambda_1\lambda_2$ if and only if $\det F > 0$, i.e.\ if $F\in\GLp(2)$. Expressing $\det F$ for all $F\in\R^{2\times 2}$ requires the signed singular values $\lambdahat_1\colonequals\lambda_1$, $\lambdahat_2\colonequals\rm{sign}(\det F)\.\lambda_2$.} of $F$. It has been shown \cite{dacorogna1988counterexample} that the function is rank-one convex if $\gamma\in[0,\frac{4}{\sqrt{3}}\approx 2.309]$, but polyconvex only if $\gamma\leq 2$. This result can be extended to the more general class
\begin{align}
	W(F)=\norm{F}^{2\alpha}\left(\norm F^2-\gamma\.\det F\right)\,,\qquad\alpha\geq 1\,,\quad\gamma\in\R\,.
\end{align}
At present, it is not known whether this expression is quasiconvex for $\gamma\in(2,\frac{4}{\sqrt{3}}]$. However, extensive numerical calculations suggest that it is quasiconvex \cite{dacorogna1998some,dacorogna1996remarks}. In the same works, Dacorogna et al.\ study several rank-one convex functions, including an example by Ball \cite{Ball84b,dacorogna1990some}:
\begin{align}
	W(F)=\norm{F}^{4\alpha}-2^{2\alpha-1}-\gamma\.(\det F)^{2\alpha}\,,\qquad\alpha\geq\frac14\,,\quad\gamma\geq0\,.
\end{align}
Together with an example by Aubert \cite{aubert1987counterexample}, the Dacorogna-Marcellini function has been the first example given in the literature of a planar function which is rank-one convex but not polyconvex. 

Many planar functions used in the context of Morrey's conjecture have the structure $W(F)=g(\norm{F}^2,\det F)$ for which additional numerical optimization is available \cite{gremaud1995numerical,grabovsky2018rank} or are composed of polynomials up to the degree four \cite{gutierrez2007optimization,bandeira2009finding}.

Guerra and da Costa \cite{guerra2020numerical} recently employed a systematic numerical approach towards the question of Morrey's conjecture. Their findings suggest that in the planar case, rank-one convexity implies $N$-wave quasiconvexity for $N\leq5$. Since any function which is $N$-wave quasiconvex for all $N\in\N$ is quasiconvex \cite[Proposition 3.6]{guerra2020numerical} (cf.~\cite{sebestyen2017laminates}), these results provide some evidence for the conjecture that rank-one convexity indeed implies quasiconvexity for planar energies.

Previous applications of machine learning in nonlinear elasticity, as we consider in Section \ref{sec:numericSid}, have mostly focused on the energy function itself \cite{fernandez2021anisotropic,agn_klein2021neural}. The application to deformation functions presented here is based on the concept of physics-informed neural networks, which have recently been employed for finding approximate solutions to various partial differential equations \cite{raissi2019physics,KarniadakisEtAl2021}.
%
%
%
%
%
%%%%%%%%%%%%%%%%%%%%%%%%%%%%%%%%%%%%%%%%%%%%%%%%%%%%%%%%%%%%%%%%%
\section{Exploitable properties of functions}\label{sec:functionProperties}

Before employing numerical methods to investigate whether a function is quasiconvex, it is generally useful to identify invariances and similar properties of the specific function that may allow to improve the efficiency of the numerical approach. In the following, we will focus on the particular energy function $\Wmp$, which exhibits three important properties that can be exploited numerically: isotropy, scaling invariance and the specific form of a volumetric-isochoric split.
%
%
%
%%%%%%%%%%%%%%%%%%%%%%%%%%%%%%%%%%%%%%%%%%%%%%%%%%%%%%%%%%%%%%%%%
\subsection{The volumetric-isochoric split}\label{sec:vol-iso}

The energy $\Wmp$ emerged from the investigation of the family of planar isotropic energies $W\col\GLp(2)\to\R$ with an additive volumetric-isochoric split
\begin{align}
	W(F)=\Wiso(F)+\Wvol(\det F)=\underbrace{\widetilde W_\textrm{iso}\bigg(\frac{F}{\sqrt{\det F}}\bigg)}_{\mathclap{\text{conformally invariant}}}+\underbrace{\Wvol(\det F)}_{\hspace{2cm}\mathclap{\text{purely area-dependent in 2D}}}.\label{eq:volIsoSplit}
\end{align}
We will motivate both the additional structure one can achieve with this type of energy functions as well as the candidate $\Wmp$ itself. By \cite[Lemma 3.1]{agn_martin2015rank}, energies of the type \eqref{eq:volIsoSplit} can be written as
\begin{align}
	W(F)=h\bigg(\frac{\lambda_1}{\lambda_2}\bigg)+f(\lambda_1\lambda_2)\,,\qquad h,f\col\Rp\to\R\,,\quad h(t)=h\bigg(\frac{1}{t}\bigg)\qquad\text{for all}\quad t\in(0,\infty)\,,
\end{align}
where $\lambda_1,\lambda_2>0$ denote the singular values of $F$ and $h,f$ are real-valued functions. In nonlinear elasticity theory, energy functions with an additive volumetric-isochoric split are widely used, primarily to model the behaviour of slightly compressible materials \cite{Ciarlet1988,agn_hartmann2003polyconvexity,ogden1978nearly,agn_neff2015geometry}.

For a further representation of $W$, we introduce the \emph{(nonlinear) distortion function} or \emph{outer distortion}
\begin{align}
	\K\col\GLp(2)\to\R\,,\qquad\K(F)\colonequals\frac{1}{2}\.\frac{\norm F^2}{\det F}=\frac{\lambda_1^2+\lambda_2^2}{2\.\lambda_1\lambda_2}=\frac{1}{2}\left(\frac{\lambda_1}{\lambda_2}+\frac{\lambda_2}{\lambda_1}\right),\label{eq:distortion}
\end{align}
where $\norm{\,.\,}$ denotes the Frobenius matrix norm with $\norm{F}^2\colonequals\sum_{i,j=1}^2 F_{ij}^2$. The distortion function $\K$ is \emph{conformally invariant}, i.e.\
\begin{equation}
	\K(aR\.\grad\varphi)=\K(\grad\varphi\.aR)=\K(\grad\varphi)\qquad \text{for all}\quad a>0\,,\;R\in\SO(2)\,.
\end{equation}
Additionally, we consider the \emph{linear distortion} or \emph{(large) dilatation}
\begin{align}
	K(F)\col\GLp(2)\to\R\,,\qquad K(F)=\frac{\opnorm{F}^2}{\det F}=\frac{\lambdamax^2}{\lambdamin\lambdamax}=\frac{\lambdamax}{\lambdamin}\,,\label{eq:dilatation}
\end{align}
where $\opnorm{F}=\sup_{\norm{\xi}=1}\norm{F\.\xi}_{\R^2}=\lmax$ denotes the operator norm (i.e.\ the largest singular value) of $F$.

We can then express every conformally invariant energy $W$ on $\GLp(2)$ as $W(F)=\hhat(K(F))=\Psi(\K(F))$ with $\hhat\,,\Psi\col[1,\infty)\to\R$ \cite{agn_martin2015rank}. Note that in general,
\begin{align}
	\K(F)=\frac{1}{2}\left(K+\frac{1}{K}\right)\qquad\iff\qquad K(F)=\K(F)+\sqrt{\K(F)^2-1}=e^{\arcosh\K(F)}.
\end{align}

In a previous paper \cite{agn_voss2021morrey} we motivated the reduction to the case $f(\det F)=c\.\log\det F$ for arbitrary additive volumetric-isochoric split energies with a newly developed rank-one convexity criterion \cite{agn_voss2019volIsLog}. More specifically, we showed that if there exists a rank-one convex energy function with an additive volumetric-isochoric split that is not quasiconvex, then we can find such a function in the set
\begin{equation}
	\Ms\colonequals\left\{W(F)=h\biggl(\frac{\lambda_1}{\lambda_2}\biggr)+c\.\log(\lambda_1\lambda_2)\;\big|\;h\col(1,\infty)\to\R\,,\;c\in\R\right\}
\end{equation}
as well. It is therefore sufficient to consider $\Ms$ instead of the general class of volumetric-isochoric split energies when discussing Morrey's conjecture.
%
%
%
%
%
%%%%%%%%%%%%%%%%%%%%%%%%%%%%%%%%%%%%%%%%%%%%%%%%%%%%%%%%%%%%%%%%%
\subsection{Scaling invariance}\label{sec:scalingInvariance}

In general an arbitrary energy $W\in\Ms$ is neither simply scaling invariant, i.e.\ $W(\alpha\.F)\neq W(F)$ for all $\alpha\in\R$ and $F\in\GLp(2)$, nor tension-compression symmetric, i.e.\ $W(F)\neq W(F\inv)$ for all $F\in\GLp(2)$. However, $\Ms$ provides additional invariance properties that hold for rank-one convexity and quasiconvexity which we will use to simplify numerical calculations in the following sections by reducing the number of deformation gradients $F\in\GLp(2)$ that we must consider.
\begin{lemma}\label{lem:rankoneScalingInvariance}
	Let $W\in\Ms$ be twice differentiable. Then the ellipticity domain of $W$ is scaling invariant, i.e.\ a cone: if $W$ is elliptic at $F_0\in\GLp(2)$, then $W$ is elliptic at $\alpha\.F_0$ for every $\alpha>0$.
\end{lemma}
\begin{proof}
	Let $\alpha>0$. The isochoric part $\Wiso(F)=h\bigl(\frac{\lambda_1}{\lambda_2}\bigr)$ of $W$ is conformally invariant by definition, which implies \mbox{$\Wiso(\alpha\.F)=\Wiso(F)$}, and therefore
	\begin{align}
		D^2\Wiso[\alpha F].(H,H) &= \frac{\mathrm{d}^2}{\mathrm{d}t^2}\.\Wiso(\alpha\. F + t\.H)\biggr|_{t=0} = \frac{\mathrm{d}^2}{\mathrm{d}t^2}\.\Wiso\bigl(\alpha\.\bigl(F+\alpha\inv t\.H\bigr)\bigr)\biggr|_{t=0}\notag\\
		&= \frac{\mathrm{d}^2}{\mathrm{d}t^2}\.\Wiso(F + t\.\alpha\inv H)\biggr|_{t=0}= D^2\Wiso(F).(\alpha\inv H,\alpha\inv H)\notag\\
		&= \frac{1}{\alpha^2}\.D^2\Wiso(F).(H,H)\label{eq:DWisoScalingInvariance}
	\intertext{for all $H\in\R^{2\times2}$. For the volumetric part $\Wvol(F)=c\.\log\det F$ we calculate}
		D^2\Wvol[\alpha\.F].(H,H) &= \frac{\mathrm{d}^2}{\mathrm{d}t^2}\. \Wvol(\alpha\. F + t\.H)\biggr|_{t=0}= \frac{\mathrm{d}^2}{\mathrm{d}t^2}\.c\.\log\det(\alpha\. F + t\.H)\biggr|_{t=0}\notag\\
		&=c\.\ddt\left[\frac{1}{\det(\alpha\. F + t\.H)}\cdot D\det[\alpha\. F + t\.H].H\right]_{t=0}\\
		&=-\frac{c}{\det(\alpha\. F)^2}\cdot \iprod{\Cof(\alpha\. F),H}^2+\frac{c}{\det(\alpha\. F)}\cdot D^2\det[\alpha\.F].(H,H)\notag
	\end{align}
	for all $H\in\R^{2\times2}$. For ellipticity of $W$, we can assume $\rank H=1$ and note that the determinant is affine linear in direction of rank-one matrices \cite[Theorem 5.20]{Dacorogna08}, and thus $D^2\det[\alpha\.F].(H,H)=0$ if $\rank(H)=1$. Therefore,
	\begin{align}
		D^2\Wvol[\alpha F].(H,H) &= -\frac{c}{\det(\alpha\. F)^2}\cdot \iprod{\Cof(\alpha F),H}^2
		= -\frac{c}{\alpha^4\det(F)^2}\cdot \iprod{\alpha\Cof(F),H}^2\notag\\
		&= -\frac{1}{\alpha^2}\cdot\frac{c}{\det(F)^2}\.\iprod{\Cof(F),H}^2
		= \frac{1}{\alpha^2}\.D^2\Wvol(F).(H,H)\,,\qquad\text{if }\rank(H)=1\,.
	\end{align}
	Together with \eqref{eq:DWisoScalingInvariance} this implies
	\begin{align}
		D^2W[\alpha F].(H,H) &= D^2\Wiso[\alpha F].(H,H) + D^2\Wvol[\alpha F].(H,H)\\
		&= \frac{1}{\alpha^2}\.D^2\Wiso(F).(H,H) + \frac{1}{\alpha^2}\.D^2\Wvol(F).(H,H)
		= \frac{1}{\alpha^2}\.D^2W(F).(H,H)\notag
	\end{align}
	for all $F,H\in\R^{2\times2}$ with $\rank(H)=1$ and all $\alpha>0$, which implies the scaling invariance of the ellipticity domain of $W\in\Ms$.
\end{proof}
\begin{lemma}\label{lem:rankoneInverseInvariance}
	Let $W\in\Ms$ be twice differentiable. The ellipticity domain of $W$ is invariant under inversion, i.e.\ if $W$ is elliptic at $F_0\in\GLp(2)$, then it is elliptic at $F_0\inv$.
\end{lemma}
\begin{proof}
	Due to the isotropy of every $W\in\Ms$, and since the singular values of $F$ and $F^T$ are identical, $W(F)=W(F^T)$ and therefore
	\begin{align}
		D^2W[F^T].(H,H) &= \frac{\mathrm{d}^2}{\mathrm{d}t^2}\.W(F^T + t\.H)\biggr|_{t=0} = \frac{\mathrm{d}^2}{\mathrm{d}t^2}\.W\bigl((F+t\.H^T)^T\bigr)\biggr|_{t=0}\notag\\
		&= \frac{\mathrm{d}^2}{\mathrm{d}t^2}\.W(F + t\.H^T)\biggr|_{t=0}=D^2\Wiso(F).(H^T,H^T)\,,
	\end{align}
	thus ellipticity at $F$ and $F^T$ are equivalent, cf.\ \cite{kruzik1999composition}. Moreover, Lemma \ref{lem:rankoneScalingInvariance} states that the ellipticity domain of $W$ is scaling invariant, i.e.\ ellipticity at $F$ implies $W$ is elliptic for all $\alpha\.F$ with $\alpha>0$.
	Combining both lemmas and choosing $\alpha=\frac{1}{\det F}>0$, we find that ellipticity at $\Cof F$ would imply ellipticity at $\frac{1}{\det F}\.(\Cof F)^T=F\inv$. Therefore, it remains to show that $W$ is elliptic at $\Cof F$.
	
	In the planar case, the singular values of $F$ and $\Cof F$ are identical\footnote{For arbitrary $F\in\R^{2\times2}$ we find
	\begin{align*}
		F\colonequals\matr{a&b\\c&d},\qquad 	B=F\.F^T=\matr{a^2+b^2&a\.c+b\.d\\a\.c+b\.d&c^2+d^2},\qquad\Cof B=\matr{c^2+d^2&-a\.c-b\.d\\-a\.c-b\.d&a^2+b^2}\\
		\implies\qquad\det(B-\lambda\.\id)=\lambda^2-(a^2+b^2+c^2+d^2)\.\lambda+\underbrace{(a^2+b^2)(c^2+d^2)-(a\.c+b\.d)^2}_{=\.\det B\.=\.\det\Cof B}=\det(\Cof B-\lambda\.\id)\,.
	\end{align*}
	Hence the eigenvalues of $B$ and $\Cof B=(\Cof F)(\Cof F)^T$ and thus the singular values of $F$ and $\Cof F$ are identical.} and thus $W(\Cof F)=W(F)$ for all $F\in\GLp(2)$ due to the isotropy of the energy. Furthermore, because of the simple shape of the cofactor matrix in the planar case $\Cof\matr{a&b\\c&d}=\matr{d&-c\\-b&a}$ we find
	\[
		\Cof(X+Y)=\Cof(X)+\Cof(Y)\qquad\text{and}\qquad\Cof(\Cof X)=X\qquad\text{for all }X,Y\in\R^{2\times2}.
	\]	
	Note carefully that these properties do \emph{not} hold for dimension $n>2$. Thus
	\begin{align*}
		D^2W[\Cof F].(H,H) &= \frac{\mathrm{d}^2}{\mathrm{d}t^2}\. W(\Cof F+t\.H) \biggr|_{t=0} = \frac{\mathrm{d}^2}{\mathrm{d}t^2}\. W(\Cof F+t\.\Cof(\Cof H)) \biggr|_{t=0}\\
		&= \frac{\mathrm{d}^2}{\mathrm{d}t^2}\. W(\Cof(F+t\.\Cof H))\biggr|_{t=0}= \frac{\mathrm{d}^2}{\mathrm{d}t^2}\. W(F+t\.\Cof H)\biggr|_{t=0}\\
		&= D^2W(F).(\Cof H,\Cof H)\,,
	\end{align*}
	which completes the proof because $\rank(H)=1$ implies\footnote{The rank of a matrix $M\in\Rnn$ is equal to the maximal size of nonzero minors. The cofactor $\Cof M$ is formed by $(n-1)\times(n-1)-$ minors of $M$ and consequently $\rank(\Cof M)=0$ if $\rank M<n-1$. In the case $\rank M=n-1$, at least one $(n-1)\times(n-1)-$ minor is nonzero which ensures $\rank(\Cof M)>0$. Because $M\.(\Cof M)^T=\det M\cdot\id=0$, the image of $(\Cof M)^T$ is in the kernel of $M$ which is of dimension $1$ and therefore $\rank(\Cof M)=1$.} $\rank(\Cof H)=1$.
\end{proof}
\begin{lemma}\label{lem:quasiScalingInvariance}
	Let $W\in\Ms$ be twice differentiable. Then the quasiconvexity domain of $W$ is scaling invariant: if $W$ is quasiconvex at $F_0\in\GLp(2)$, i.e.\ if
	\begin{equation}
		\int_\Omega W(F_0+\nabla\vartheta(x))\,\dx\geq\int_\Omega W(F_0)\,\dx\qquad\text{for all}\quad\vartheta\subset W_0^{1,\infty}(\Omega;\Rn)
	\end{equation}
	for any domain $\Omega\subset\R^2$,	then the energy is quasiconvex at $\alpha\.F_0$ for all $\alpha>0$.
\end{lemma}
\begin{proof}
	We show that the so-called \emph{energy gap}
	\begin{equation}
		\int_\Omega W(F_0+\nabla\vartheta(x))\,\dx-|\Omega|\.W(F_0)\,,\label{eq:energyGap}
	\end{equation}
	i.e.\ the difference between the energy of $\varphi(x)=F_0\.x+\vartheta(x)$ and the homogeneous solution $\varphi_0(x)=F_0\.x$, is scaling invariant. Note that quasiconvexity at $F_0$,
	\begin{align*}
		\int_\Omega W(\nabla\varphi(x))\,\dx\geq |\Omega|\.W(F_0)\qquad\text{if}\quad\varphi(x)|_{\partial\Omega}=F_0.x\,,
	\end{align*}
	implies that the energy gap is always non-negative. We write $F=\nabla\varphi$ and compute
	\begin{align}
		\int_\Omega W(\alpha\.F)\,\dx-|\Omega|\cdot W(\alpha\.F_0)&=\int_\Omega W(\alpha\.F)-W(\alpha\.F_0)\,\dx\notag\\
		&=\int_\Omega\Wiso(\alpha\.F)+c\.\log\det(\alpha F)-\Wiso(\alpha\.F_0)-c\.\log\det(\alpha F_0)\,\dx\notag\\
		&=\int_\Omega\Wiso(F)-\Wiso(F_0)+c\.\log(\alpha^2\det F)-c\.\log(\alpha^2\det F_0)\dx\notag\\
		&=\int_\Omega \Wiso(F)-\Wiso(F_0)+c\left(2\.\log\alpha+\log\det F-2\.\log\alpha-\log\det F_0\right)\dx\notag\\
		&=\int_\Omega\Wiso(F)+c\.\log\det F-\Wiso(F_0)-c\.\log\det F_0\,\dx\\
		&=\int_\Omega W(F)-W(F_0)\,\dx=\int_\Omega W(F)\,\dx-|\Omega|\cdot W(F_0)\,.\notag\qedhere
	\end{align}
\end{proof}
\begin{remark}
	With Lemmas \ref{lem:rankoneScalingInvariance} and \ref{lem:quasiScalingInvariance} we showed that both rank-one convexity and quasiconvexity are scaling invariant for any energy $W\in\Ms$. Regarding Morrey's question for planar isotropic energies with volumetric-isochoric split we can therefore assume $\det F_0=1$ without loss of generality. More specifically, for arbitrary $W\in\Ms$ we just need to prove the rank-one convexity for all $F_0\in\GLp(2)$ with $\det F_0=1$ to obtain rank-one convexity at all $F\in\GLp(2)$. Likewise, it is sufficient to only check for quasiconvexity starting with a homogeneous deformation $x\mapsto F_0\.x$ with $\det F_0=1$ in place of arbitrary $F\in\GLp(2)$, cf.\ Figure \ref{fig:notRankOneconvex}.
\end{remark}
\begin{figure}[h!]
	\centering
	\includegraphics[width=.49\textwidth]{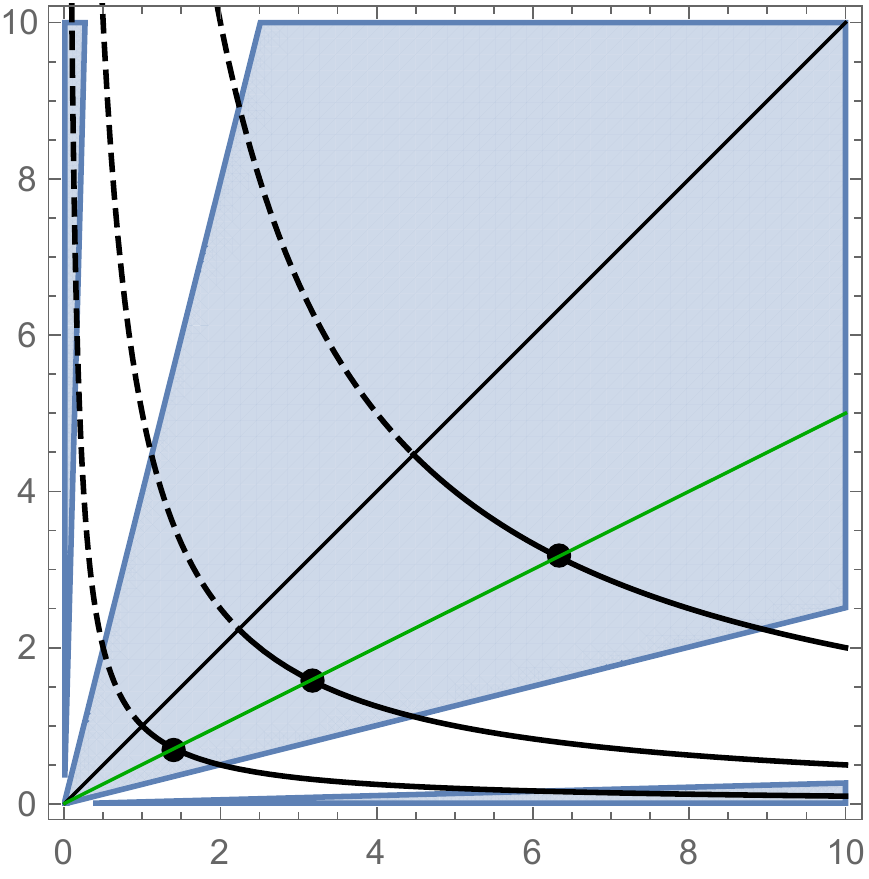}
  	\hfill
	\includegraphics[width=.49\textwidth]{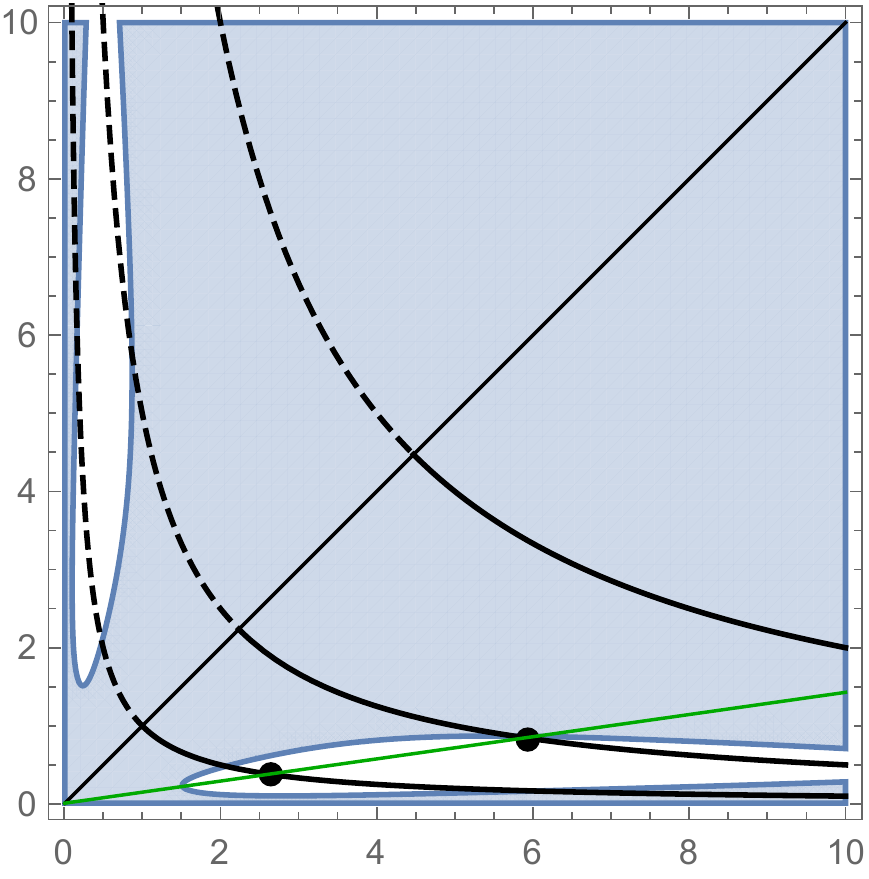}
	\caption{Visualization of a possible [ellipticity/quasiconvexity] domain (shaded blue) of a planar isotropic energy function in terms of the singular values $\lambda_1,\lambda_2>0$. Left: With scaling invariance for $W\in\Ms$, i.e.\ [ellipticity/quasiconvexity] on the curve $\det F=\lambda_1\lambda_2=c$ (black) implies [ellipticity/quasiconvexity] for the corresponding cone. Right: For an arbitrary volumetric-isochoric split energy this invariance is lost.}\label{fig:notRankOneconvex}
\end{figure}
%
%
%
%
%
%%%%%%%%%%%%%%%%%%%%%%%%%%%%%%%%%%
\subsection{The least rank-one convex candidate $\Wmp(F)$}\label{sec:Wmp}

Continuing with the class $\Ms$, i.e.\ with energy functions with $c\log\det F$ as volumetric part, we focus on positive $c>0$. As multiplying a function by a scalar does not change its convexity behavior, we can consider $c=1$ without loss of generality and assume that the isochoric part is convex.\footnote{If the isochoric part $h(t)$ is not convex, the combined function $h\bigl(\frac{\lambda_1}{\lambda_2}\bigr)+\log(\lambda_1\lambda_2)$ will not correspond to a rank-one convex function \cite{agn_voss2021morrey}.} Thus we consider the class
\begin{equation}
	\Mps\colonequals\left\{W(F)=h\biggl(\frac{\lambda_1}{\lambda_2}\biggr)+\log(\lambda_1\lambda_2)\;|\;h\col(1,\infty)\to\R \text{ is convex}\right\}.
\end{equation}
Using sharp rank-one convexity conditions \cite{agn_voss2019volIsLog}, it is possible to identify \enquote{least} rank-one convex candidates by searching for functions that satisfy those conditions by equality. Surprisingly \cite{agn_voss2021morrey}, it is possible to show  that the question of quasiconvexity reduces to the single energy function
\begin{align}
	\Wmp(F)&=\frac{\lambdamax}{\lambdamin}-\log\frac{\lambdamax}{\lambdamin}+\log\det F=\frac{\lambdamax}{\lambdamin}+2\.\log\lambdamin\label{eq:WmpinK}\\
	&=\K(F)+\sqrt{\K(F)^2-1}-\arcosh\K(F)+\log\det F\,.\notag
\end{align}
Hence, if $\Wmp$ were quasiconvex, then every function in the class $\Mps$ and thus every rank-one convex planar isotropic energy function with an additive volumetric-isochoric split whose isochoric part $h(t)$ is convex would be quasiconvex as well. A first analytic observation (see Appendix \ref{sec:RadialEulerLagrange}) could not conclusively answer this question, but opens the possibility to interesting microstructures having the same energy value as the homogeneous deformation. This motivates to proceed numerically and try to falsify the quasiconvexity of $\Wmp$.
%
%
%
%
%
%%%%%%%%%%%%%%%%%%%%%%%%%%%%%%%%%%%%%%%%%%%%%%%%%%%%%%%%%%%%%%%%%
\section{The classical finite element approach}\label{sec:numericOliver}

A first way to show that a given energy density is not quasiconvex is the finite element method. For this, we discretize the displacement field $\vartheta$ by Lagrange finite elements \cite{ciarlet2002finite}. We use triangle grids and first-order finite elements only. That way, the deformation gradient and hence the integrand are piecewise constant and the hyperelastic energy $\int_\Omega W(\nabla\varphi)\,\dx$ can be computed without quadrature error which is important in view of exactly calculating the energy gap \eqref{eq:energyGap}. Our implementation is based on the \textsc{Dune} libraries~\cite{bastian_et_al:2021,sander:2020}.
%
%
%
%
%
%%%%%%%%%%%%%%%%%%%%%%%%%%%%%%%%%%%%%%%%%%%%%%%%%%%%%%%%%%%%%%%%%
\subsection{Testing for quasiconvexity}

We perform tests on two domains: The square $[-1,1]^2$ and the unit disc $B_1(0)$. Both are filled with coarse triangle grids as shown in Figure \ref{fig:circle_domain_and_coarse_grid}. For ease of implementation, we approximate the boundary of the disc by six quadratic arcs (dashed lines). The finite elements grids are then constructed by uniform refinement of the coarse grids. For the disc grid, new boundary vertices are placed not at edge midpoints but on the curved arcs approximation the boundary. The final grids consist of 16641 vertices and 32768 elements for the square and 12481 vertices and 24576 elements for the disc. Due to the scaling invariance (the results of Section \ref{sec:scalingInvariance}) it is sufficient to test with an $F_0$ such that $\det F_0=1$, thus due to isotropy we consider only $F_0=\diag(\sqrt{a},\frac{1}{\sqrt{a}})$ with arbitrary $a>0$. For the result shown here we select $a=2$.

\begin{figure}[h!]
	\begin{center}
		\begin{tikzpicture}[scale=2.5]
			\coordinate (a1)  at ( -1, 0);
			\coordinate (a2)  at (  1, 0);
			\coordinate (a3)  at (  0, 0);
			\coordinate (a4)  at (  0,-1);
			\coordinate (a5)  at (  0, 1);
			\coordinate (a6)  at ( -1,-1);
			\coordinate (a7)  at (  1, 1);
			\coordinate (a8)  at (  1,-1);
			\coordinate (a9)  at ( -1, 1);
			
			% The grid
			\draw (a6) -- (a4) -- (a8) -- (a2) -- (a7) -- (a5) -- (a9) -- (a1) -- (a6) -- (a3) -- (a4) -- (a2) -- (a3) -- (a5) -- (a1) -- (a3) -- (a7);
		\end{tikzpicture}
		\hspace{2.5cm}
		\begin{tikzpicture}[scale=2.5]
			\coordinate (a1)  at ( -1, 0);
			\coordinate (a2)  at (  1, 0);
			\coordinate (a3)  at (  0, 0);
			\coordinate (a4)  at ( -0.5, -0.8660254037844386);
			\coordinate (a5)  at (  0.5, -0.8660254037844386);
			\coordinate (a9)  at (  0.5, 0.8660254037844386);
			\coordinate (a10) at ( -0.5, 0.8660254037844386);
			
			% The grid
			\draw (a1) -- (a10) -- (a9) -- (a2) -- (a5) -- (a4) -- cycle;
			\draw (a1) -- (a2);
			\draw (a10) -- (a5);
			\draw (a9) -- (a4);
			
			\draw [dashed] (a3) circle (1);
		\end{tikzpicture}
	\end{center}
	\caption{Coarse grids for square and disc domain. As the disc grid gets refined, it approximates the piecewise polynomial boundary better and better.}\label{fig:circle_domain_and_coarse_grid}
\end{figure}
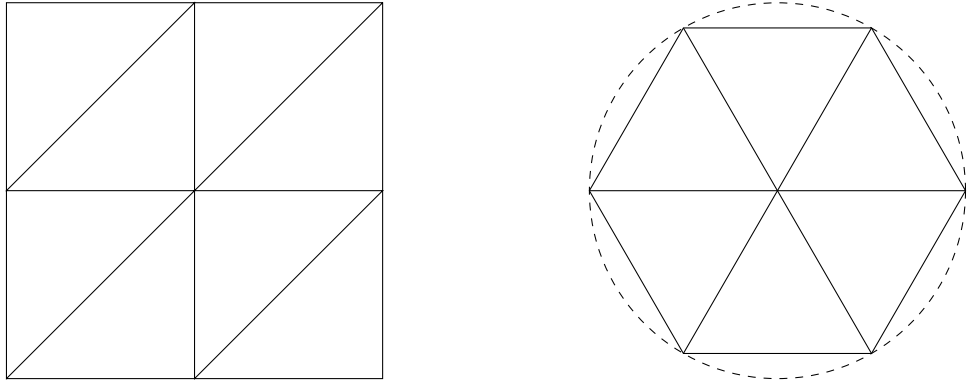

We minimize the hyperelastic energy $\int_\Omega\Wmp(F)\,\dx$ with a trust-region algorithm. These algorithms have been thoroughly studied in the literature, and they can be shown to always converge to stationary points of the energy~\cite{conn_gould_toint:2000}. As trust-region methods are descent methods, maximizers and saddle-points are not attractive points, and convergence is therefore typically towards (local) minimizers only.

Trust-region methods perform sequences of quadratic minimization problems with convex inequality constraints. We use a trust-region defined in terms of the maximum norm, and therefore the convex constraints are a set of independent bound constraints. For the quadratic bound-constrained minimization problem we then use a monotone multigrid method~\cite{kornhuber:1994} as suggested in \cite{sander:2012}. Such methods achieve multigrid convergence rates even for bound-constrained problems. We solve each inner problem until the maximum-norm of the correction drops below $10^{-5}$. The large but sparse tangent matrices are computed using the ADOL-C algorithmic differentiation system~\cite{walther_griewank:2012}. 

When looking for global minimizers with a descent method, the question of initial iterates is of central importance. As shown exemplary in the next section, when the energy is not quasiconvex, imperfections caused by the finite-precision arithmetic are sufficient to drive the system towards microstructures even starting from the homogeneous configuration. For the particular energy $\Wmp$, however, this did not lead to any energy decrease. We obtained the same negative results for some other \enquote{obvious} initial iterates, such as random perturbations of the homogeneous state. More involved constructions of non-homogeneous initial iterates are described in Sections \ref{sec:numericNonElliptic} and \ref{sec:Curl}.
%
%
%
%
%
%%%%%%%%%%%%%%%%%%%%%%%%%%%%%%%%%%%%%%%%%%%%%%%%%%%%%%%%%%%%%%%%%%%%%
\subsection{Non-elliptic microstructure}\label{sec:numericNonElliptic}

In the following, we introduce several additional ideas to search for microstructures with energy levels below the homogeneous state with more adept methods.

In order to better understand the shape of a possible microstructure, we ensured its existence by considering slightly modified problems. We start with the weakened energy
\begin{equation}
	W_c(F)=\frac{\lambdamax}{\lambdamin}-\log\frac{\lambdamax}{\lambdamin}+c\.\log\det F\,,\qquad c>1\,,\label{eq:Wc}
\end{equation}
which is not rank-one convex anymore but satisfies $\displaystyle\lim_{c\to1}W_c(F)=\Wmp(F)$ for all $F\in\GLp(2)$. We are interested in the resulting microstructures especially if they do not appear to be simple laminations caused by the loss of ellipticity \cite{ball1987,dolzmann1999numerical,li2000finite}. Any local minimizer found for $W_c$ can then be used as an initial deformation for minimizing $\Wmp$ again with the hope to maintain the non-homogeneous structure and thereby disprove quasiconvexity of the energy $\Wmp$ itself.

\begin{figure}[h!]
	\centering
	\includegraphics[width=.47\textwidth, trim = 5.3cm 0.2cm 0.3cm 3.5cm, clip]{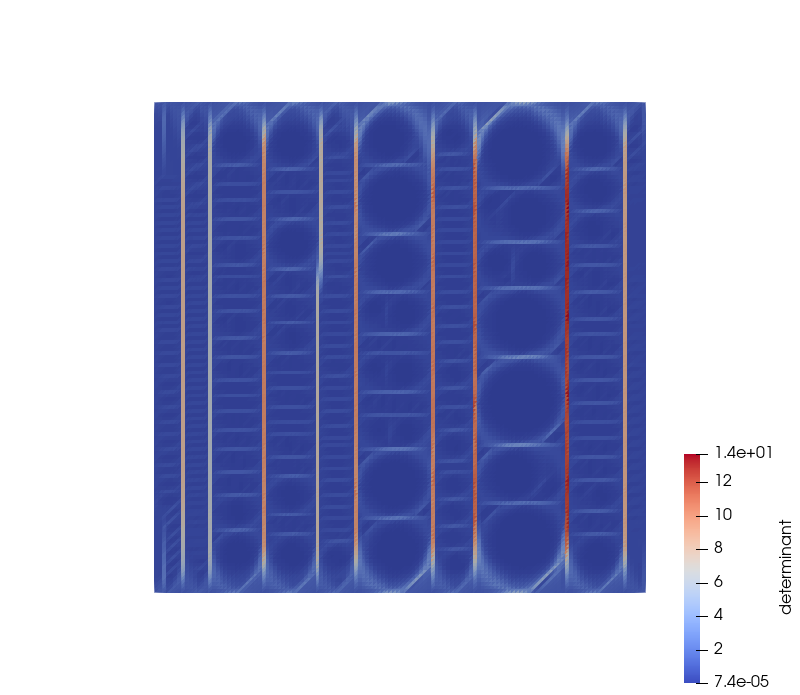}
	\hfill
	\includegraphics[width=.47\textwidth, trim = 5.3cm 0.2cm 0.3cm 3.5cm, clip]{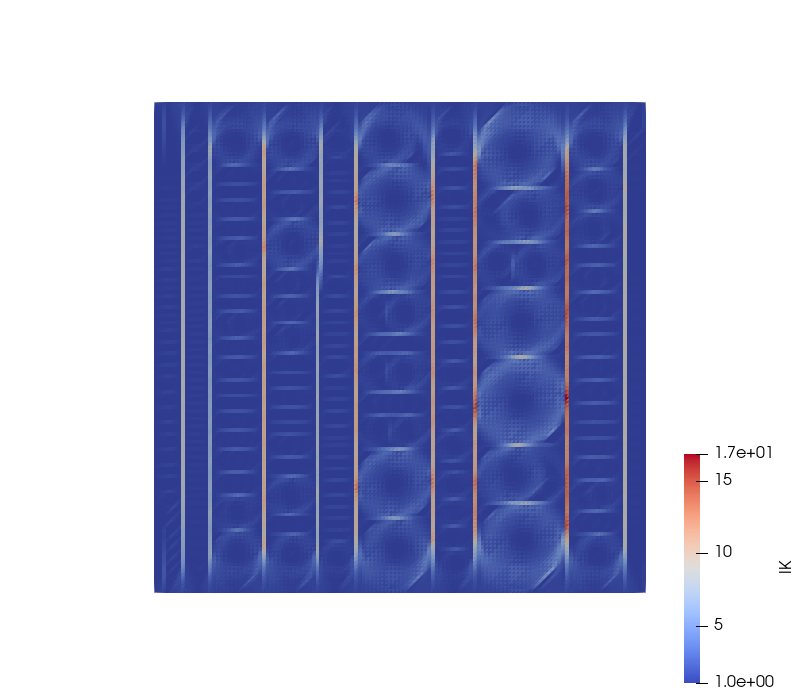}
	\caption{Microstructure for the non-elliptic energy $W_c$ with $c=1.1$ and $F_0=\diag\bigl(\sqrt{2},\frac{1}{\sqrt{2}}\bigr)$. The color shows the determinant (left) and the distortion $\K(F)=\frac{1}{2}\.\frac{\norm F^2}{\det F}$ (right).}\label{fig:LaminateMicrostructure}
\end{figure}

While our material $W_c$ indeed shows microstructures that contain a lamination structure, we observe radially symmetric contracting regions as well, cf.\ Figure \ref{fig:LaminateMicrostructure}. However, we already know that deformations of this kind cannot lower the energy value of $\Wmp$ (see Appendix \ref{sec:RadialEulerLagrange}) and thus they are not suited for finding a new microstructure by using them as an initial configuration for minimizing $\Wmp(F)$ again. This is confirmed by direct numerical experiments: when letting the trust-region solver of the previous section minimize $\Wmp$ starting from the configurations shown in Figure \ref{fig:LaminateMicrostructure}, all we obtain is convergence to the homogeneous state.

Therefore, we continue with an alternative numerical experiment to produce more convoluted microstructures. For this we place three disjoint balls $B_{r_i}(x_i)$, $i=1,2,3$ with radius $0.2$ and center $x_1=(-0.5,0)$, $x_2=(0.35,0.35)$, and $x_3=(0.35,-0.35)$ inside the unit disc domain. We then set $c>1$ inside each circles $B_{r_i}(x_i)$ but fix $c=1$ elsewhere, i.e.
\[
	c(x)=\begin{cases}c^\star&\caseif x\in\cup_i B_{r_i}(x_i)\,,\\1&\caseelse\end{cases}\qquad\text{with}\quad c^\star>1\,.
\]
\begin{figure}[h!]
	\centering
	\includegraphics[width=.47\textwidth, trim = 5.3cm 0.2cm 0.3cm 3.5cm, clip]{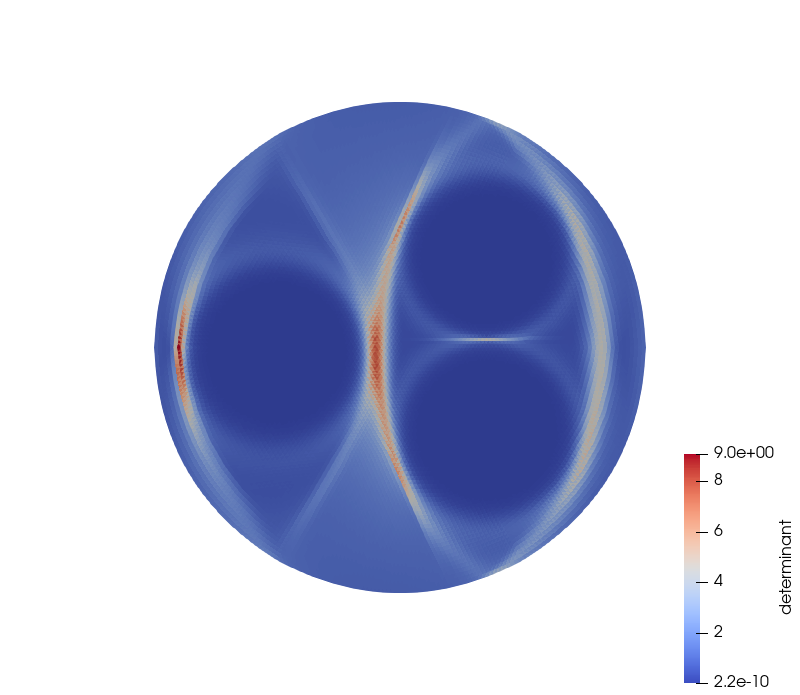}
	\hfill
	\includegraphics[width=.47\textwidth, trim = 5.3cm 0.2cm 0.3cm 3.5cm, clip]{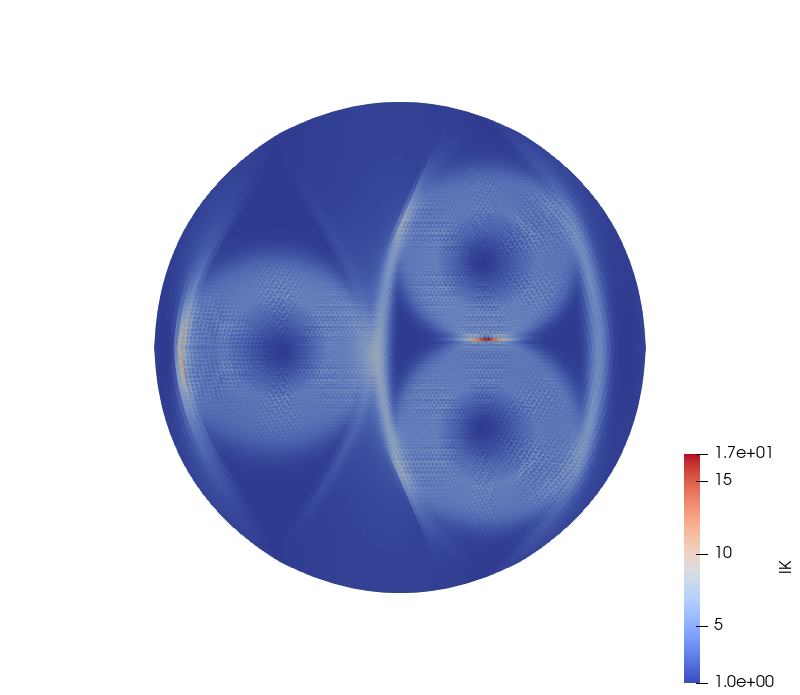}
	\includegraphics[width=.47\textwidth, trim = 5.3cm 0.2cm 0.3cm 3.5cm, clip]{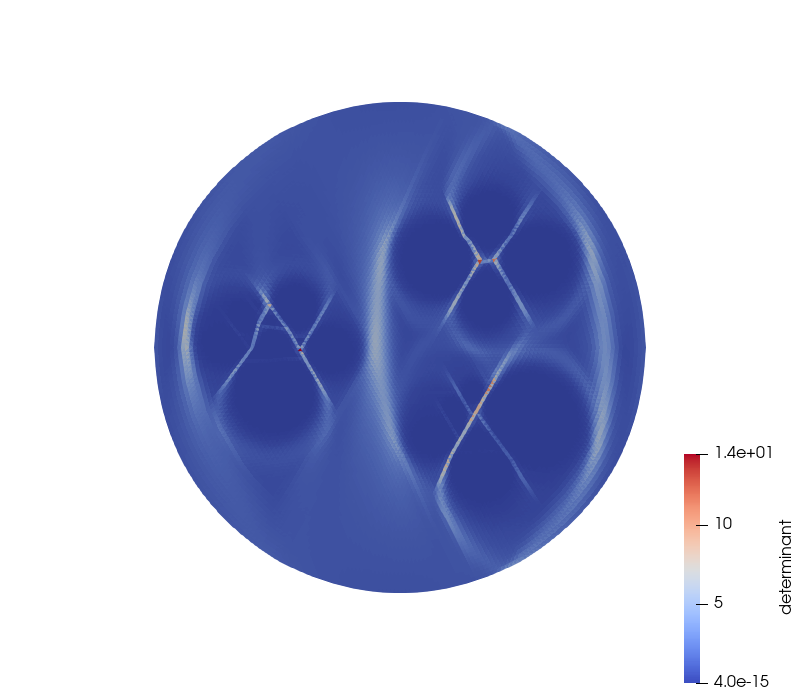}
	\hfill
	\includegraphics[width=.47\textwidth, trim = 5.3cm 0.2cm 0.3cm 3.5cm, clip]{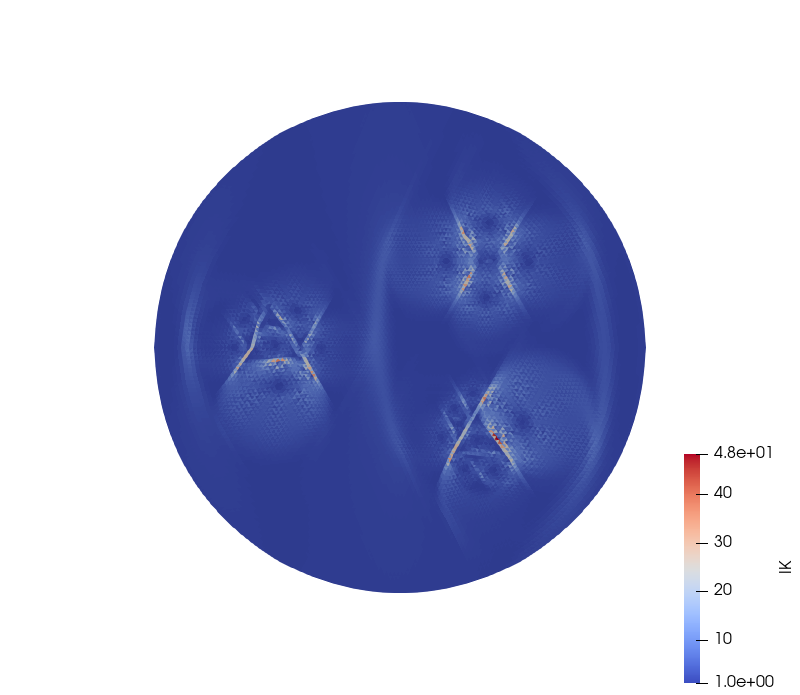}
	\caption{Microstructures for the non-elliptic energy $W_c$ with $c=1.1$ (top) and $c=2$ (bottom) inside three smaller circles and $c=1$ outside starting. The boundary deformation is $F_0=\diag\bigl(\sqrt{2},\frac{1}{\sqrt{2}}\bigr)$. The colors encode the determinant (left) and the distortion $\K(F)=\frac{1}{2}\.\frac{\norm F^2}{\det F}$ (right).}\label{fig:CircleMicrostructure2}
\end{figure}

For values slightly larger than $c=1$ we observe that all three circles contract in a radially symmetric fashion, cf.\ Figure \ref{fig:CircleMicrostructure2}. Again, we note that Appendix \ref{sec:RadialEulerLagrange} shows that radial symmetric contracting deformations have the same energy value for the limit case $c=1$. Increasing the weighting of $\log\det F$ by raising $c$ results in lower energy values compared to the homogeneous deformation. For higher values of $c$, the microstructure becomes more convoluted but keeps its contraction structure, cf.\ Figure \ref{fig:CircleMicrostructure2}.

We note that both microstructures are primarily located inside these circles (where ellipticity is lost), while for the rest of the material with $c=1$, the deformation remains mostly homogeneous. In particular, the borders of the inner circles maintain their shape to a certain extent, even though we do not impose additional internal boundary conditions to ensure this. We interpret these observations as a first indicator towards our energy candidate $\Wmp$ being quasiconvex since the response of $\Wmp$ is \enquote{stable} toward the assumed non-elliptic perturbation in the circle.
%
%
%
%
%
%%%%%%%%%%%%%%%%%%%%%%%%%%%%%%%%%%%%%%%%%%%%%%%%%%%%%%%%%%%%%%%%%
\section{Deep neural networks}\label{sec:periodic}

One may conjecture that minimizing $\Wmp$ in a finite element space fails to find microstructure because each finite element coefficient influences only a very local part of the deformation function. Additionally, when the  strain energy density lacks quasiconvexity (which we consider possible for $W_\text{magic}^+$), FEM-methods generally fail \cite{KumarEtAl2020} due to the non-uniqueness of the solution, i.e. the microstructures.
In this chapter, we experiment with an alternative approach where this relationship is more global.
%
%
%
%
%
%%%%%%%%%%%%%%%%%%%%%%%%%%%%%%%%%%%%%%%%%%%%%%%%%%%%%%%%%
\subsection{Physics-informed neural networks}\label{sec:numericSid}

We use a numerical scheme which is based on deep neural networks as an ansatz for solving partial differential equations, an idea also referred to  as \textit{physics-informed neural networks} \cite{raissi2019physics,KarniadakisEtAl2021}. In principle, this is similar to the ansatz constructed classically using finite element functions. However, deep neural networks generally lead to highly nonlinear and more efficient (in terms of number of parameters) approximation spaces with considerable approximation properties even for low numbers of coefficients.

In the following, we consider periodic deformations only, cf.\ Section \ref{sec:periodicBoundary}. Consider the following ansatz for the periodic superposition:
\begin{align}
	\vartheta_{\#,\omega}(x_1,x_2) =\;&(1-\cos(2\pi x_1))\.\mathcal{F}_{\omega_f}(x_1)+(1-\cos(2\pi x_2))\.\mathcal{G}_{\omega_g}(x_2)\notag\\
		&+(1-\cos(2\pi x_1)) (1-\cos(2\pi x_2))\.\mathcal{H}_{\omega_h} (x_1,x_2)\,,\label{eq:NNansatz}
\end{align}
where $\mathcal{F}_{\omega_f},\mathcal{G}_{\omega_g}\col[0,1]\rightarrow \R^2$ and $\mathcal{H}_{\omega_h}\col[0,1]^2\rightarrow \R^2$ are deep neural networks \cite{Schmidhuber2015} parameterized by parameters $\omega_f$, $\omega_g$, and $\omega_h$, respectively. The ansatz is intentionally constructed this way to identically satisfy the periodic boundary conditions \eqref{eq:periodicBounderayConditions}. The neural network architectures are given by repeated composition of successive high-dimensional linear and nonlinear transformations as
\begin{align}
		\mathcal{F}_{\omega_f} (x_1) &=
		\mathcal{L}^{64\rightarrow2} _{\omega_{f,5}} \circ
		\mathcal{R}\circ
		\mathcal{L}^{64\rightarrow64} _{\omega_{f,4}} \circ
		\mathcal{R}\circ
		\mathcal{L}^{64\rightarrow64} _{\omega_{f,3}} \circ
		\mathcal{R}\circ
		\mathcal{L}^{64\rightarrow64} _{\omega_{f,2}} \circ
		\mathcal{R}\circ
		\mathcal{L}^{1\rightarrow64} _{\omega_{f,1}}
		(x_1)\,,\notag\\   
		\mathcal{G}_{\omega_g} (x_2) &=
		\mathcal{L}^{64\rightarrow2} _{\omega_{g,5}} \circ
		\mathcal{R}\circ
		\mathcal{L}^{64\rightarrow64} _{\omega_{g,4}} \circ
		\mathcal{R}\circ
		\mathcal{L}^{64\rightarrow64} _{\omega_{g,3}} \circ
		\mathcal{R}\circ
		\mathcal{L}^{64\rightarrow64} _{\omega_{g,2}} \circ
		\mathcal{R}\circ
		\mathcal{L}^{1\rightarrow64} _{\omega_{g,1}}
		(x_2)\,,\\
		\mathcal{H}_{\omega_h} (x_1,x_2) &=
		\mathcal{L}^{64\rightarrow2} _{\omega_{h,5}} \circ
		\mathcal{R}\circ
		\mathcal{L}^{64\rightarrow64} _{\omega_{h,4}} \circ
		\mathcal{R}\circ
		\mathcal{L}^{64\rightarrow64} _{\omega_{h,3}} \circ
		\mathcal{R}\circ
		\mathcal{L}^{64\rightarrow64} _{\omega_{h,2}} \circ
		\mathcal{R}\circ
		\mathcal{L}^{2\rightarrow64} _{\omega_{h,1}}
		(x_1,x_2)\,.\notag
\end{align}
Here, $\mathcal{L}^{i \rightarrow j}_{\omega_{\square,k}}$, $k=1,\cdots,5$ ($\square$ is a placeholder for `$f$', `$g$', and `$h$') denotes the $k^\text{th}$-linear transformation parameterized by the set of weights and biases $\omega_{\square,k} = \{A_{\square,k},b_{\square,k}\}$ (with $\omega_{\square} = \{\omega_{\square,k}\}$) such that any $z\in\R^i$ is transformed via
\begin{equation}
	\mathcal{L}^{i \rightarrow j}_{\omega_{\square,k}} (z) = A_{\square,k}z + b_{\square,k}\,, \qquad \text{with} \qquad A_{\square,k}\in\R^{j \times i},\; b_{\square,k}\in\R^j\,.
\end{equation}
The linear transformations are interleaved with element-wise nonlinear transformations $\mathcal{R}(\cdot) = \tanh(\cdot)$.

Since each layer of a neural network is a differentiable operation, the  gradient of the superposition field $\vartheta_\#$ can be computed using the chain rule. This is efficiently implemented using an automatic differentiation engine \cite{AtilimEtAl2018}. Note that, unlike numerical differentiation by, e.g., finite differences, the gradients computed via automatic differentiation are analytically exact.

For numerical integration of the strain energy density over the domain $\Omega$ we discretize the domain with a uniform grid  $\{(x_1^\alpha,x_2^\beta)\,|\, \alpha,\beta=1,\dots,N\}$ of $N \times N$ points.  For a given $\nabla \vartheta_{\#,\omega}$, the total strain energy is approximated via the trapezoidal integration rule as
\begin{equation}
	I(\varphi)=	\int_\Omega\Wmp(F_0+\nabla\vartheta_{\#,\omega}(x))\,\dx \approx   \sum_{\alpha=1}^N \sum_{\beta=1}^N \xi(\alpha)\xi(\beta)\.\Wmp\bigl(F_0+\nabla\vartheta_{\#,\omega}(x_1^\alpha,x_2^\beta)\bigr),
\end{equation}
where the integration weights $\xi(\alpha)$ are
\begin{equation}
	\xi(\alpha) = \begin{cases}
		\frac12\.\frac{1}{N-1}&\caseif\; \alpha = 1 \text{ or } \alpha = N\,,\\
		\frac{1}{N-1} &\caseelse[.]
	\end{cases}
\end{equation}
Whether the trapezoidal integration rule over- or underestimates the integral depends on whether the integrand is convex or concave, respectively, in the interval of the integration. However, if the integrand exhibits an inflection point (which is also observed here), over-/underestimation of the integral cannot be guaranteed with this rule. 

The optimal parameters $\omega=\{\omega_f,\omega_g,\omega_h\}$ of the neural networks are then obtained as minimizers of the total energy
\begin{equation}
	\omega^\star = \arg \min_\omega \sum_{\alpha=1}^N \sum_{\beta=1}^N \xi(\alpha) \xi(\beta)\. W_\text{magic}^{+}\bigl(F_0+\nabla\vartheta_{\#,\omega}(x_1^\alpha,x_2^\beta)\bigr).
\end{equation}
The minimization problem is solved iteratively using \emph{Adam} \cite{KingmaEtAl2017}, an efficient first-order gradient-based stochastic optimization method. The derivatives of the objective function with respect to the parameters $\omega$ are computed via automatic differentiation again. Following the minimization, the periodic superposition field $\vartheta_\#=\vartheta_{\#,\omega^\star}$ is given by \eqref{eq:NNansatz}. The \emph{Adam} optimizer was used for $2000$ iterations with an initial learning rate of $10^{-3}$ which was decayed by a factor of $0.1$ after the $700^{\text{th}}$, $1400^{\text{th}}$, and $1800^{\text{th}}$ iterations. The numerical scheme was implemented in PyTorch \cite{PyTorch2019}.

\figurename ~\ref{fig:numericsPeriodic} illustrates the representative microstructures obtained via the numerical scheme for $F_0$ equal to $\diag\bigl(3,\frac{1}{3}\bigr)$ and $\diag\bigl(10,\frac{1}{10}\bigr)$ on a grid  of resolution $N=128$.  For both values of $F_0$ the microstructure has the form of a \emph{smooth laminate} and its energy seems to equal the one of the corresponding homogeneous deformation up to machine precision\footnote{Due to the underestimation of integration using the trapezoidal rule, the total strain energy was observed to be slightly below the energy of the homogeneous deformation. However, under grid-refinement with $N\gg 1$, this error converges to zero and the total strain energy was observed to equal the energy via homogeneous deformation up to machine precision.} which motivates the search for a precise form of the analytical solution.

\begin{figure}[h!]
	\centering
	\begin{subfigure}{\textwidth}
		\centering
		\includegraphics[width=\textwidth]{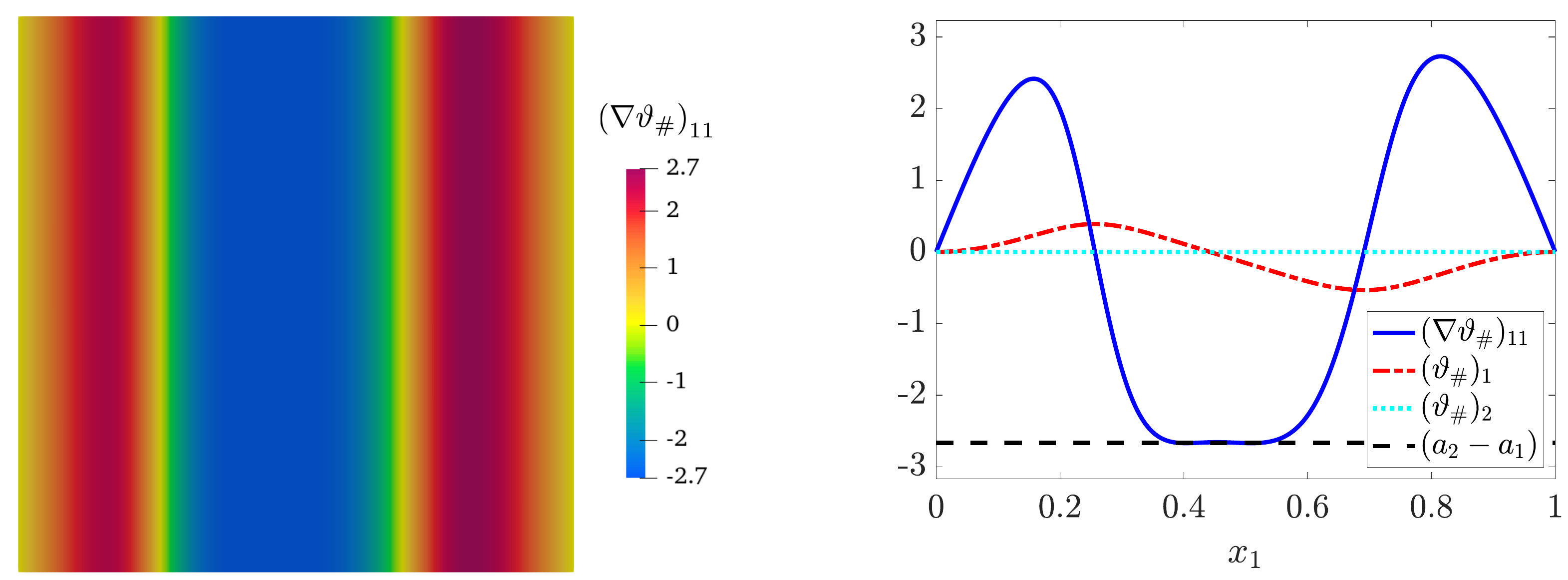}
		\caption{$F_0 = \diag(a_1,a_2),\ a_1=3,\ a_2=1/3$}
	\end{subfigure}
	\begin{subfigure}{\textwidth}
		\centering
		\includegraphics[width=\textwidth]{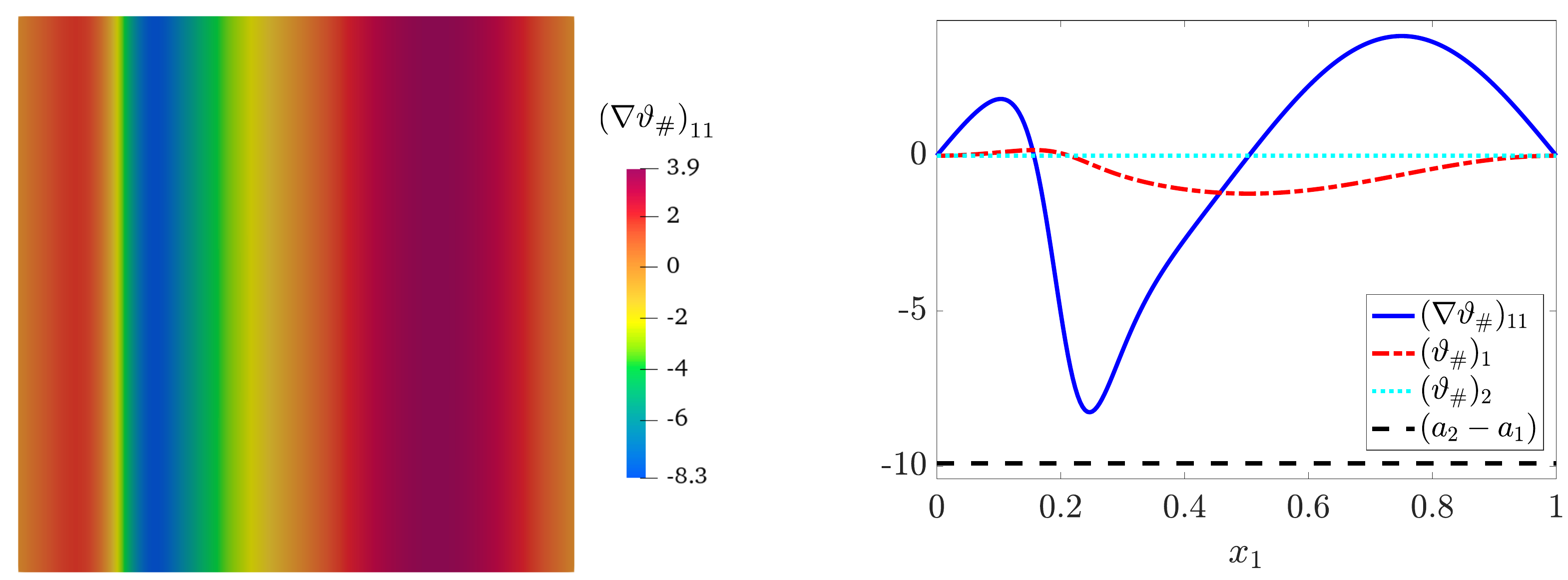}
		\caption{$F_0 = \diag(a_1,a_2),\ a_1=10,\ a_2=1/10$}
	\end{subfigure}
	\caption{Neural network minimizers of the energy $\Wmp$ under periodic boundary conditions with two different $F_0$. (Left) The microstructures are visualized using $(\nabla\vartheta_\#)_{11}$. (Right) $(\nabla\vartheta_\#)_{11}$, $(\vartheta_\#)_{1}$, and $(\vartheta_\#)_{2}$ are plotted against $x_1$ for constant $x_2=0.5$. The relevance of $a_2-a_1$ is discussed in Section \ref{sec:smoothLaminates}.\label{fig:numericsPeriodic}
	}  
\end{figure}
%
%
%
%
%
%%%%%%%%%%%%%%%%%%%%%%%%%%%%%%%%%%%%%%%%%%%%%%%%%%%%%%%%%
\subsection{Smooth laminates}\label{sec:smoothLaminates}
Guided by the numerical findings of the previous section, it turned out that the smooth laminates can be explained analytically as well.

\begin{lemma}\label{lemma:smoothLaminates}
	Let $\Omega=[0,1]^2$ be the unit square and consider \mbox{$\Wmp(F)=\frac{\lmax}{\lmin}+2\.\log\lmin$} with the ordered singular values $\lmax\geq\lmin$ of $F\in\GLp(2)$. For any homogeneous deformation \mbox{$\varphi_0(x)=F_0\.x$} with $F_0=\diag(a_1,a_2)$ and $a_1\geq a_2>0$, the elastic energy \mbox{$I(\varphi)=\int_\Omega\Wmp(\nabla\varphi)\,\dx$} is equal to $I(\varphi_0)$ for all periodical deformations $\varphi(x)=F_0\.x+\vartheta_\#(x)$ of the type
	\begin{equation}
		\vartheta_\#(x_1,x_2)=\matr{f_\#(x_1)\\0}\qquad\text{with}\qquad f_\#'(x)\geq a_2-a_1\quad\forall\;x\in [0,1]\,.\label{eq:periodicBounderaySolution}
	\end{equation}
\end{lemma}
\begin{proof}
	Let $\vartheta_\#$ be as in \eqref{eq:periodicBounderaySolution}. We find that
	\[
		F=F_0+\nabla\vartheta_\#(x)=\matr{a_1+f'_\#(x_1)&0\\0&a_2}
	\]
	is diagonal with $a_1+f'_\#(x_1)\geq a_2$. This implies $\lmax=a_1+f'_\#(x_1)$ and $\lmin=a_2$. Periodic boundary conditions as defined by \eqref{eq:periodicBounderayConditions} imply $f(0)=f(1)$ as well as $f'(0)=f'(1)$. Thus
	\begin{align}
		I(\varphi)&=\int_\Omega\Wmp(F_0\.x+\vartheta_\#(x))\,\dx=\int_\Omega\frac{\lmax}{\lmin}+2\.\log\lmin\,\dx\notag\\
		&=\int_0^1\int_0^1\frac{a_1+f'_\#(x_1)}{a_2}+2\.\log(a_2)\,\dx_1\dx_2=\frac{a_1+f_\#(1)-f_\#(0)}{a_2}+2\.\log(a_2)\\
		&=\frac{a_1}{a_2}+2\.\log(a_2)=\int_0^1\int_0^1\frac{a_1}{a_2}+2\.\log(a_2)\,\dx_1\dx_2=\int_\Omega\Wmp(F_0\.x)\,\dx=I(\varphi_0)\,.\notag\qedhere
	\end{align}
\end{proof}
We also show explicitly that any $\varphi$ as defined in Lemma \ref{lemma:smoothLaminates} is indeed an equilibrium point of $I$ by direct computation. For the corresponding Euler-Lagrange equations we must compute the first Piola-Kirchhoff stress $S_1(F)=D\Wmp(F)$. We start by considering the restriction to deformations of the type \eqref{eq:periodicBounderaySolution} as an a priori constraint and verify that the single resulting reduced Euler-Lagrange equation holds: Since
\begin{equation}
	\dd{x_1}\frac{\partial\Wmp}{\partial f_\#'}=\dd{x_1}\frac{\partial}{\partial f_\#'}\left[\frac{a_1+f'_\#(x_1)}{a_2}+2\.\log(a_2)\right]=\dd{x_1}\frac{1}{a_2}=0\
\end{equation}
in the class of deformations of the type \eqref{eq:periodicBounderaySolution}, all functions are stationary points of this constrained problem. This is a necessary condition for the general problem \cite{agn_voss2018more,agn_voss2021anti}. It remains to show that any such function is also a stationary point of the full Euler-Lagrange equations $S_1(F)=D\Wmp(F)$ as well. For this we identify
\begin{equation}
	\Wmp(F)=\Psi(\K(F))+\log\det F=\K(F)+\sqrt{\K(F)^2-1}-\arcosh\K(F)+\log\det F\,,\label{eq:WmpPsiNotation}
\end{equation}
and start with the first derivative of the nonlinear distortion function $\K(F)$:
\begin{align}
	\iprod{D\K(F),H}&=\iprod{D_F\left[\frac{\norm{F}^2}{2\.\det F}\right],H}=\frac{1}{2}\.\frac{2\.\iprod{F,H}\.\det F-\norm F^2\iprod{\Cof F,H}}{(\det F)^2}\notag\\
	&=\frac{1}{\det F}\.\iprod{F,H}-\frac{\norm F^2}{2\det F}\.\iprod{\frac{\det F\.F^{-T}}{\det F},H}=\iprod{\frac{1}{\det F}\.F-\K(F)\.F^{-T},H}.
\end{align}
Furthermore, using the notation from \eqref{eq:WmpPsiNotation}, we make use of
\begin{align}
	\Psi(t)&=t+\sqrt{t^2-1}-\arcosh t\,,\notag\\
	\Psi'(t)&=1+\frac{t}{\sqrt{t^2-1}}-\frac{1}{\sqrt{t^2-1}}=1+\sqrt{\frac{t-1}{t+1}}\,,\\
	\Psi''(t)&=\frac{1}{2}\.\sqrt{\frac{t+1}{t-1}}\frac{t+1-(t-1)}{(t+1)^2}=\frac{1}{(t+1)^2}\.\sqrt{\frac{t+1}{t-1}}\notag
\end{align}
as well as
\begin{align}
	\iprod{D[\log\det F],H}&=\iprod{\frac{1}{\det F}\.\Cof F,H}=\iprod{F^{-T},H}\,,\\
	D\left[F^{-T}\right].H&=\left(D\left[F\inv\right].H\right)^T=\left(-F\inv H\.F\inv\right)^T=-F^{-T}H^TF^{-T}.
\end{align}
Altogether, we find
\begin{align}
	\iprod{S_1(F),H}&=\iprod{D\Wmp(F),H}=f'(\K(F))\.\iprod{D\K(F),H}+\iprod{D[\log\det F],H},\notag
\intertext{which implies} S_1(F)&=\left(1+\sqrt{\frac{\K(F)-1}{\K(F)+1}}\right)\left(\frac{1}{\det F}\.F-\K(F)\.F^{-T}\right)+F^{-T}\label{eq:WmpS1}\\
	&=\frac{2\.K}{(K+1)\.\det F}\.F-\frac{K\.(K-1)}{K+1}\.F^{-T},\qquad K=\frac{\lmax}{\lmin}.\notag
\end{align}
Next, we insert deformations of the type \eqref{eq:periodicBounderaySolution} as an a posteriori constraint \cite{agn_voss2018more}:
\begin{align}
	F&=\matr{a_1+f'_\#(x_1)&0\\0&a_2},& F^{-T}&=\matr{\frac{1}{a_1+f'_\#(x_1)}&0\\0&\frac{1}{a_2}},\notag\\
	K&=\frac{a_1+f'_\#(x_1)}{a_2}\,,&\det F&=(a_1+f'_\#(x_1))\.a_2\,.
\end{align}
Thus we arrive at
\begin{align}
	S_1(F)=\matr{\frac{1}{a_2}&0\\0&\frac{a_1-2\.a_2+f'_\#(x_1)}{a_2^2}}.
\end{align}
The full Euler-Lagrange equations are therefore satisfied since
\[
	\Div D\Wmp(F)=\matr{\frac{\partial}{\partial x_1}\frac{1}{a_2}\\\frac{\partial}{\partial x_2}\frac{a_1-2\.a_2+f'_\#(x_1)}{a_2^2}}=\matr{0\\0}.\qedhere
\]
Note that the case $f_\#\equiv 0$ corresponds to the homogeneous solution and that this is the only superposition which also complies with Dirichlet boundary conditions given by $F_0$. As deformations of the type \eqref{eq:periodicBounderaySolution} only allow for a smooth displacement in one coordinate direction, we refer to them as \emph{smooth laminates}, cf.\ Figure \ref{fig:numericsPeriodic}.

\begin{figure}[h!]
  	\centering
  	\begin{minipage}[b]{.49\linewidth}
    \centering
    \begin{tikzpicture}
		\begin{axis}[
        axis x line=middle,axis y line=middle,
        x label style={at={(current axis.right of origin)},anchor=north, below},
        xlabel=$x$, ylabel={},
        xmin=-0.1, xmax=1.1,
        ymin=-0.15, ymax=1.5,
        width=1.1\linewidth,
        height=0.9\linewidth,
        ytick={0.53,0.77,1.3},
        yticklabels={$a_1-a_2$,$a_2$,$a_1$},
	    xtick={1}
        ]
        \addplot[black, dashed][domain=0:1, samples=\sample]{0.77*x} node[above,pos=0.8,yshift = 0.1cm] {$a_2\.x$};
        \addplot[black, dashed][domain=0:1, samples=\sample]{0.77*(x-1)+1.3} node[above,pos=0.8,xshift=-0.8cm,yshift = 0.1cm] {$a_2\.(x-1)+a_1$};
        \addplot[red, smooth][domain=0:1, samples=\sample]{1.3*((0.4-(x-0.5)^2)/8*sin(360*x)+x)} node[above,pos=0.5,yshift=0.4cm] {$f_\#(x)+a_1\.x$};
        \addplot[udedarkblue, dashed][domain=0:1, samples=\sample]{1.3*((0.4-(x-0.5)^2)/8*sin(360*x)+x)-0.77*x} node[above,pos=0.55] {$f_\#(x)+(a_1-a_2)\.x$};
        \addplot[udedarkblue, smooth][domain=0:1, samples=\sample]{1.3*((0.4-(x-0.5)^2)/8*sin(360*x)+x)-1.3*x} node[above,pos=0.7,yshift = 0.2cm] {$f_\#(x)$};
        \end{axis}
  \end{tikzpicture}
  \end{minipage}
  \begin{minipage}[b]{.49\linewidth}
  \centering
  \begin{tikzpicture}
    	\def\Aconstant{(1.3-0.77)/(2*1.3-0.77)}
		\begin{axis}[
        axis x line=middle,axis y line=middle,
        x label style={at={(current axis.right of origin)},anchor=north, below},
        xlabel=$x$, ylabel={},
        xmin=-0.1, xmax=1.1,
        ymin=-0.15, ymax=1.5,
        width=1.1\linewidth,
        height=0.9\linewidth,
        ytick={0.53,0.77,1.3},
        yticklabels={$a_1-a_2$,$a_2$,$a_1$},
	    xtick={1}
        ]
        \addplot[black, dashed][domain=0:1, samples=\sample]{0.77*x} node[above,pos=0.85,yshift = 0.1cm] {$a_2\.x$};
        \addplot[black, dashed][domain=0:1, samples=\sample]{0.77*(x-1)+1.3} node[above,pos=0.8,xshift=-0.8cm,yshift = 0.1cm] {$a_2\.(x-1)+a_1$};
        \addplot[red, smooth][domain=0:\Aconstant, samples=\sample]{1.3*2*x};
        \addplot[red, smooth][domain=\Aconstant:1, samples=\sample]{0.77*(x-1)+1.3} node[below,pos=0.7,yshift=-0.4cm] {$f_\#(x)+a_1\.x$};
        \addplot[udedarkblue, dashed][domain=0:\Aconstant, samples=\sample]{(1.3*2-0.77)*x};
        \addplot[udedarkblue, dashed][domain=\Aconstant:1, samples=\sample]{(1.3-0.77} node[above,pos=0.35] {$f_\#(x)+(a_1-a_2)\.x$};
        \addplot[udedarkblue, smooth][domain=0:\Aconstant, samples=\sample]{1.3*x};
        \addplot[udedarkblue, smooth][domain=\Aconstant:1, samples=\sample]{(0.77-1.3)*(x-1)} node[above,pos=0.7,yshift = 0.1cm] {$f_\#(x)$};
        \end{axis}
  \end{tikzpicture}
  \end{minipage}
  \caption{Left: Visualization of $f_\#(x)+a_1\.x\col[0,1]\to[0,a_1]$ (red line) of the type \eqref{eq:periodicBounderaySolution}, i.e.\ $f'_\#(x)\geq a_2-a_1$ for all $x\in[0,1]$. Right: Visualization of a two-phase laminate as described in Remark \eqref{remark:periodicLaminate}.}\label{fig:laminate}
\end{figure}
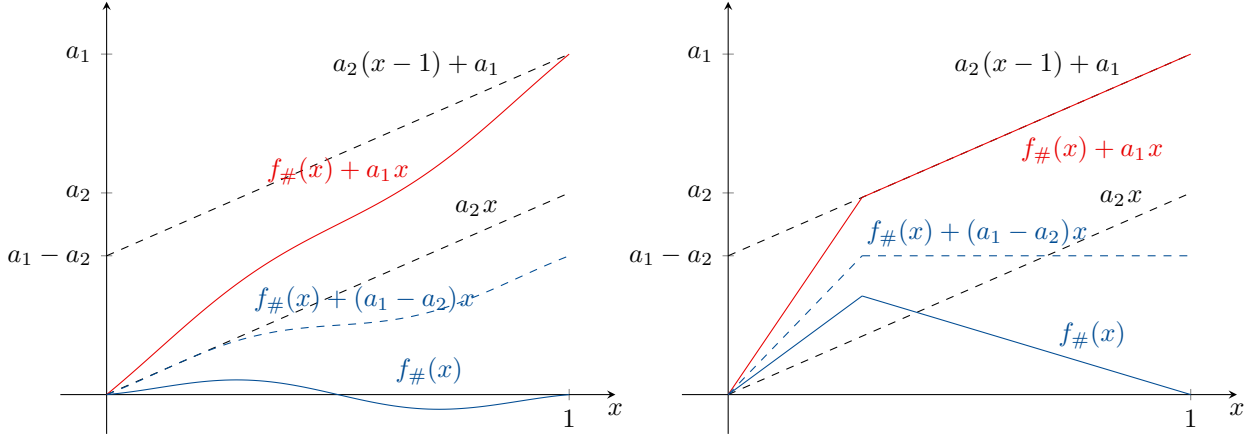

\begin{remark}\label{corollary:smoothLaminates}
	We can construct a continuous map remaining in the class of smooth laminates (i.e.\ which only contains stationary points of $\Wmp$) from one candidate $f_\#$ to the identity $f_\#\equiv 0$ and from there to another solution $\ftilde_\#$ at constant energy value $I(\varphi)$. Therefore, neither the homogeneous solution nor any of these smooth laminates are stable. A similar argument holds for the radially symmetric deformations considered in \cite{agn_voss2021morrey}.
\end{remark}
\begin{remark}\label{remark:periodicLaminate}
	If we allow $\vartheta_\#$ to be non-smooth (cf.\ Figure \ref{fig:laminate}) we can obtain various (classical) lamination patterns as well. For example, as a limit of a smooth laminate microstructure of the type \eqref{eq:periodicBounderaySolution}, we may consider
	\begin{equation}
		f_\#(x)=\begin{cases}a_1\.x &\caseif x\in\bigl[0,\frac{a_1-a_2}{2\.a_1-a_2}\bigr],\\[
		0.5em]
		(a_2-a_1)\.(x-1) &\caseif x\in\bigl[\frac{a_1-a_2}{2\.a_1-a_2},1\bigr],\end{cases}
	\end{equation}
	which yields the simple two-phase laminate
	\begin{align}
		F_1&=\matr{2\.a_1 &0\\0& a_2},\qquad F_2=\matr{a_2 &0\\0& a_2},\\[.7em]
		F_0&=\frac{a_1-a_2}{2\.a_1-a_2}\.F_1+\frac{a_1}{2\.a_1-a_2}\.F_2=\matr{a_1 &0\\0& a_2}.\notag
	\end{align}
	Recall that we can only construct laminations that satisfy the constraint $f_\#'(x)+a_1\geq-a_2$, cf.\ Lemma \ref{lemma:smoothLaminates}.
\end{remark}
%
%
%
%
%
%%%%%%%%%%%%%%%%%%%%%%%%%%%%%%%%%%%%%%%%%%%%%%%%%%%%%%%%%%%%%%%%%
\section{Relaxation to non-gradient fields with a Curl-based penalty}\label{sec:Curl}

While the smooth laminates discussed in Section \ref{sec:periodic} as well as the contracting deformations in \cite{agn_voss2021morrey} (cf.\ Appendix \ref{sec:RadialEulerLagrange}) allow for non-homogeneous deformations whose energy level for $\Wmp$ is as low as the homogeneous one, we have yet failed to find strictly lower energy values.

For a new attempt, we expand our set of possible deformations $\varphi$. To this end, we extend the energy functional from gradient fields $F=\nabla\varphi$ to more general matrix fields $P$, but control the distance of $P$ to the set of compatible mappings (i.e.\ gradient fields) with the $\Curld$ operator. This approach is commonly used when dealing with local dislocations in gradient plasticity with plastic spin \cite{agn_munch2008curl,agn_neff2009notes} and with relaxed micromorphic continua \cite{agn_neff2014unifying}. We expect to obtain new microstructures numerically and, in order to gain additional insight into our original variational problem, we will observe the material behavior as a function of the weight parameter on the penalty term $\Curld P$. The hope is that when such new microstructures are used as initial iterates, the optimization algorithm of Section \ref{sec:numericOliver} will converge to an energy level below the homogeneous one.

More specifically, instead of the classical minimization problem in the space $W^{1,2}(\Omega,\R^2)$, i.e.\
\begin{equation}
	I_1(\varphi)\colonequals\int_\Omega\Wmp(\nabla\varphi)\,\dx\to\min\,,\qquad\varphi(x)|_{\partial\Omega}=F_0.x\label{eq:curlI1}
\end{equation}
with $\Wmp$ as in \eqref{eq:WmpinK}, we consider
\begin{equation}
	I_2(P)\colonequals\int_\Omega\Wmp(P)+\frac{L_c^2}{2}\.\norm{\Curld P}^2\,\dx\to\min\,,\qquad P.\tau|_{\partial\Omega}=F_0.\tau\.,\label{eq:curlI2}
\end{equation}
in the larger space $H(\Curl)$. Here, the vector field $\tau$ is the unit tangent vector to $\partial\Omega$\footnote{In the three-dimensional case we can use $P\times\nu|_{\partial\Omega}=F_0\times\nu$ with $\nu\in\R^3$ as the unit normal vector to $\partial\Omega$ as equivalent alternative, since $\displaystyle(P-F_0)\times\nu|_{\partial\Omega}=0\iff(P-F_0).\tau_{1/2}|_{\partial\Omega}=0$ where $\tau_1,\tau_2$ are the tangent vectors to $\partial\Omega$.} and $L_c\in\Rp$ is a penalty parameter. The planar operator $\Curld$ is discussed in Appendix \ref{sec:appendixCurl}. Note that the mapping $P\col\Omega\subset\R^2\to\R^{2\times 2}$ does not need to be a gradient field, i.e.\ $P$ may be \emph{incompatible}, but that
\begin{equation}
	\inf I_2(P)\leq \inf I_2(\nabla\varphi)=\inf I_1(\varphi) \qquad\text{and}\qquad\lim_{L_c\to\infty} I_2(P) = +\infty \quad\text{if $P$ is not a gradient field}\,,
\end{equation}
since we are considering contractible domains only, and on such domains $\Curld P=0$ if and only if $P=\grad\varphi$ for some weakly differentiable $\varphi$. 

We minimize the functional $I_2$ of \eqref{eq:curlI2} numerically by approximating the matrix field $P$ by piecewise polynomial functions $P_n$ such that each row of $P_n$ is a first-order N\'ed\'elec finite element of the first kind \cite{kirby2012common} which are elements of the space $H(\Curl)$ by construction. As the domain, we use the unit ball $\Omega=B_1(0)$ and the same grid as in Section \ref{sec:numericNonElliptic}. The algorithm used to minimize the discretized functional is the same trust-region multigrid algorithm as in Section \ref{sec:numericNonElliptic} as well.

\begin{figure}[h!]
	\centering
	\includegraphics[width=.47\textwidth, trim = 5.3cm 0.2cm 0.3cm 3.5cm, clip]{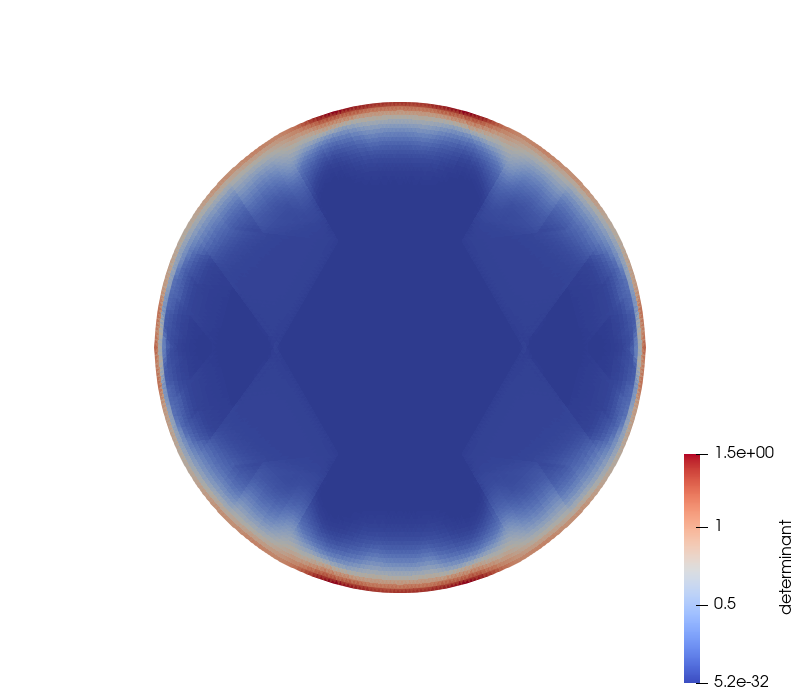}
	\hfill
	\includegraphics[width=.47\textwidth, trim = 5.3cm 0.2cm 0.3cm 3.5cm, clip]{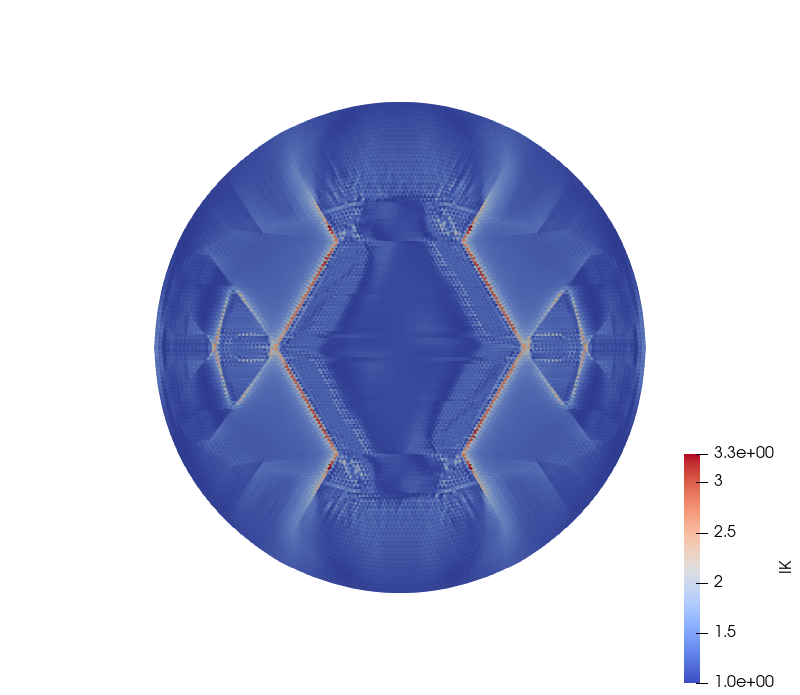}
	\caption{Matrix field $\Phat$ computed for $I_2(P)\to\min$ and $L_c=0.5$ with a starting deformation of $F_0=\diag\bigl(\sqrt{2},\frac{1}{\sqrt{2}}\bigr)$. The colors encode the determinant (left) and the distortion $\K(\Phat)=\frac{1}{2}\.\frac{\norm\Phat^2}{\det\Phat}$ (right).}\label{fig:Curl1}
\end{figure}

We consider the minimization problem \eqref{eq:curlI2} for different parameters $L_c$ and choose the boundary deformation $F_0=\diag\bigl(\sqrt{2},\frac{1}{\sqrt{2}}\bigr)$. We must note that for the range of parameters $L_c$ considered here, the trust-region algorithm of Section \ref{sec:numericOliver} would not converge to a stationary point. Rather, due to the severe degeneracy of the energy function \eqref{eq:curlI2} for low $L_c$, the solver would get stuck eventually. In such situations, the trust-region control would decrease the trust-region radius further and further, without ever finding an acceptable new iterate. In view of the finite-precision arithmetic used in actual simulation, this is not a contradiction to the trust-region theory, which claims that such a new iterate will always be found.

All figures show the last step computed. For $L_c=0.5,1,2$ we observe different non-trivial microstructures with significantly lower energy value than the homogeneous solution, cf.\ Figure \ref{fig:Curl1}. The larger $L_c$, the closer the energy approaches the energy of the homogeneous state.

After calculating $\Phat$ from $I_2(P)\to\min$ by the above numerical method, we construct a deformation field $\widehat\varphi\col\Omega\to\R^2$ with a deformation gradient close to $\Phat$ by computing
\begin{equation}
	\widehat\varphi\colonequals\argmin_\varphi\norm{\nabla\varphi-\Phat}_{L^2(\Omega)}^2=\int_\Omega\norm{\nabla\varphi-\Phat}^2\,\dx\,,\qquad\varphi(x)|_{\partial\Omega}=F_0.x\,,
\end{equation}
in the space of first-order Lagrange finite elements. While the compatible deformations $\widehat\varphi$ all look similar to their corresponding incompatible matrix fields $\Phat$, cf.\ Figure \ref{fig:Curl2}, the energy value $I_1(\widehat\varphi)$ is always higher then the homogeneous one $I_1(\varphi_0)$. Starting the minimization algorithm for $I_1$ from $\widehat\varphi$ always results in the homogeneous configuration.

\begin{figure}[h!]
	\centering
	\includegraphics[width=.47\textwidth, trim = 5.3cm 0.2cm 0.3cm 3.5cm, clip]{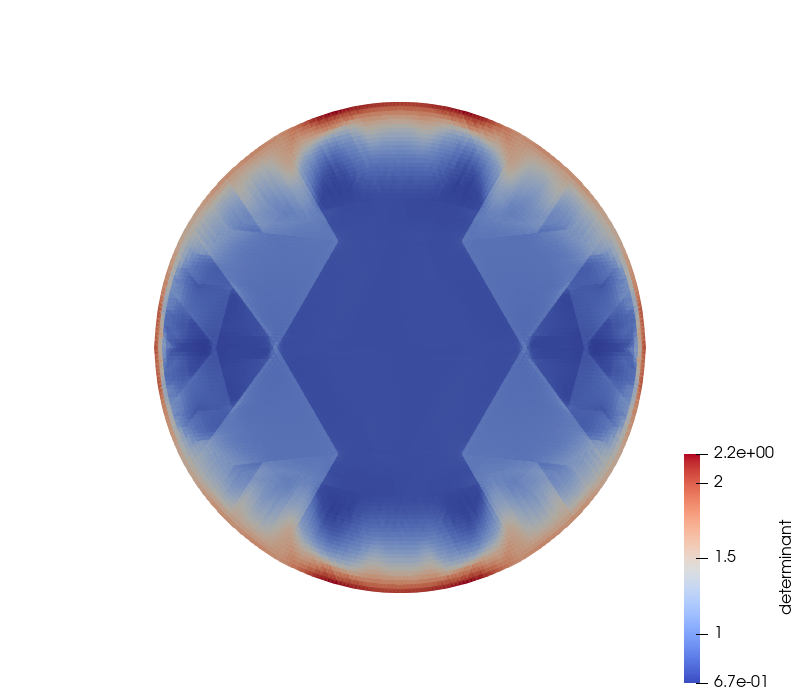}
	\hfill
	\includegraphics[width=.47\textwidth, trim = 5.3cm 0.2cm 0.3cm 3.5cm, clip]{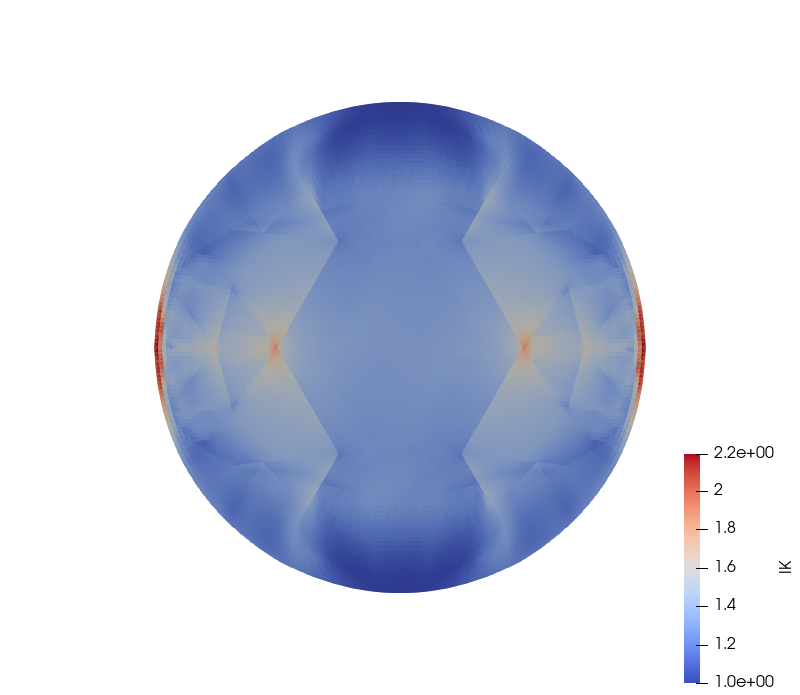}
	\caption{Compatible deformation $\widehat\varphi$ from $\norm{\nabla\varphi-\Phat}_{L^2(\Omega)}\to\min$ for the matrix field $\Phat$ shown in the Figure \ref{fig:Curl1}. Colors show determinant (left) and the distortion $\K(\nabla\widehat\varphi)=\frac{1}{2}\.\frac{\norm{\nabla\widehat\varphi}^2}{\det\nabla\widehat\varphi}$ (right).}
	\label{fig:Curl2}
\end{figure}

As a further attempt to reach low values of $I_2(P)$, we also started the optimization from a non-compatible matrix field instead of the homogeneous deformation gradient $P=F_0$ we used for the numerical methods before. For this we chose a checkerboard pattern with squares of size $\frac1b\times\frac1b$ and alternately set the values $F_1=(1-\delta)\.F_0$ and $F_2=(1+\delta)\.F_0$ with $\delta=0.5$, cf.\ Figure \ref{fig:Curl3}.

\begin{figure}[h!]
	\centering
	\includegraphics[width=.47\textwidth, trim = 5.3cm 0.2cm 0.3cm 3.5cm, clip]{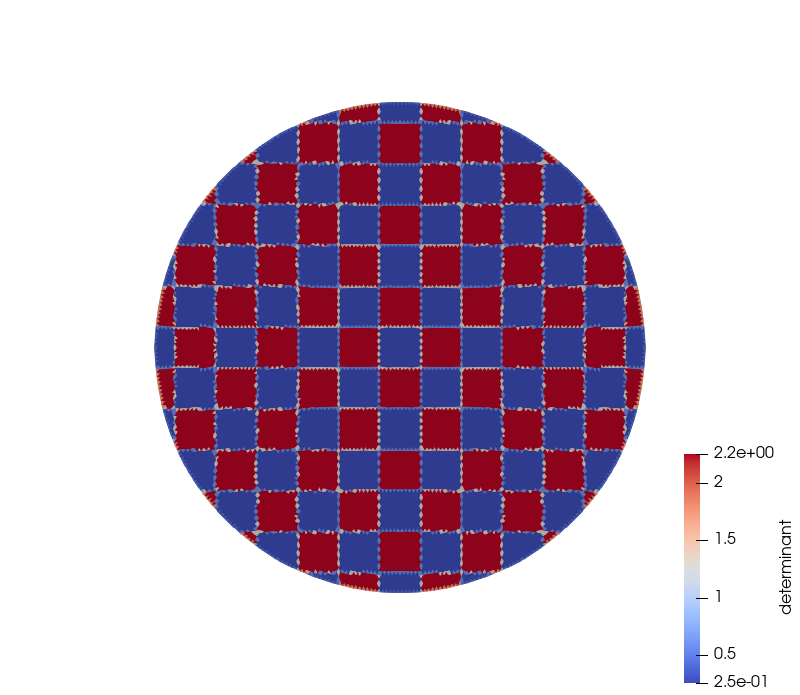}
	\hfill
	\includegraphics[width=.47\textwidth, trim = 5.3cm 0.2cm 0.3cm 3.5cm, clip]{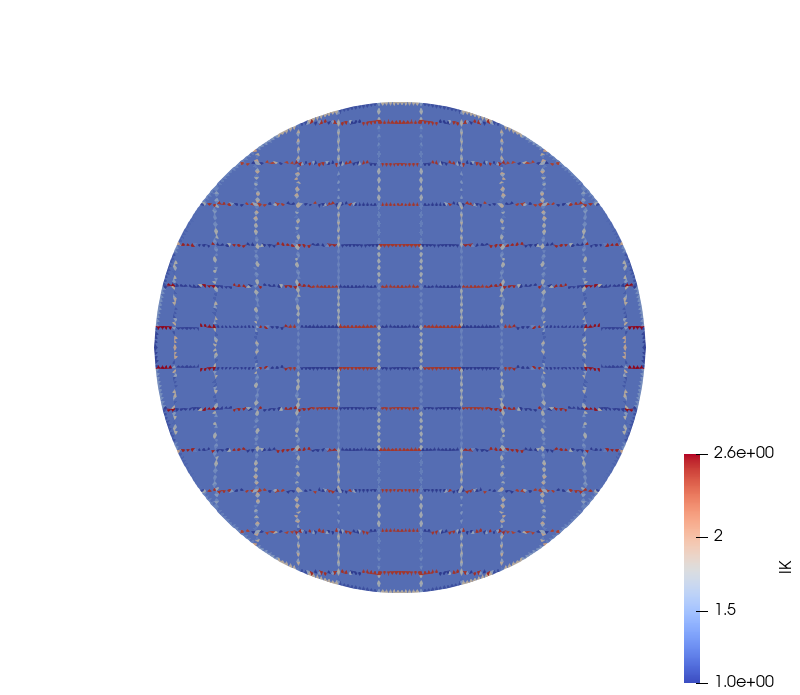}
	\caption{Checkerboard pattern used as initial iterate. The pattern is not a gradient and combines an inhomogeneous determinant with a constant distortion $\K$.}
	\label{fig:Curl3}
\end{figure}

This pattern has a lower energy value without considering the regularization term $\norm{\Curl P}^2$ as it only activates the concave volumetric part $\log\det F$ of $\Wmp(F)$. As shown in Figure \ref{fig:Curl4}, we indeed arrive at different microstructures. However, the compatible deformations obtained from minimizing $	\norm{\nabla\varphi-\Phat}_{L^2(\Omega)}^2$ again have higher energy values than the homogeneous state. Furthermore, minimizing the original energy $I_1$ from there lead back to the homogeneous configuration once again.

\begin{figure}[h!]
	\centering
	\includegraphics[width=.47\textwidth, trim = 5.3cm 0.2cm 0.3cm 3.5cm, clip]{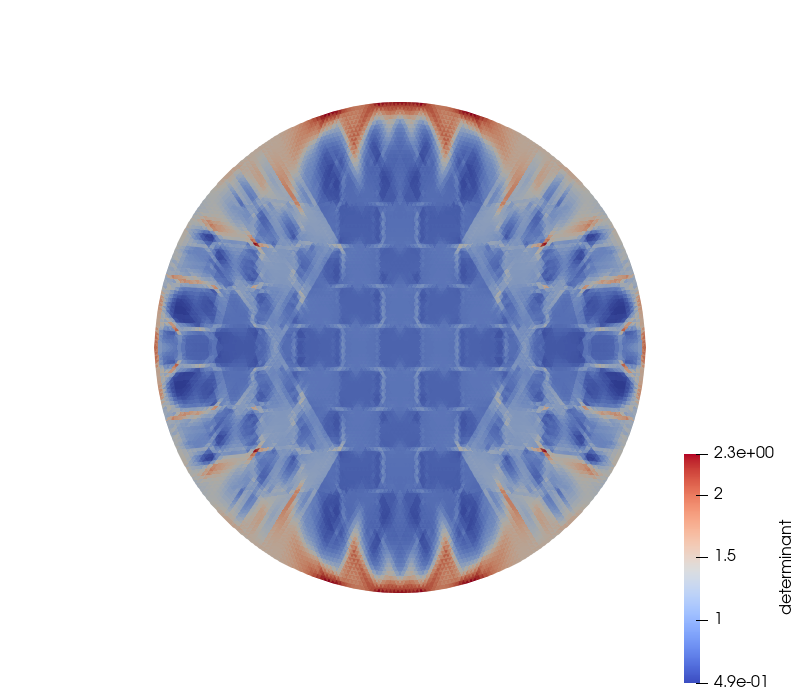}
	\hfill
	\includegraphics[width=.47\textwidth, trim = 5.3cm 0.2cm 0.3cm 3.5cm, clip]{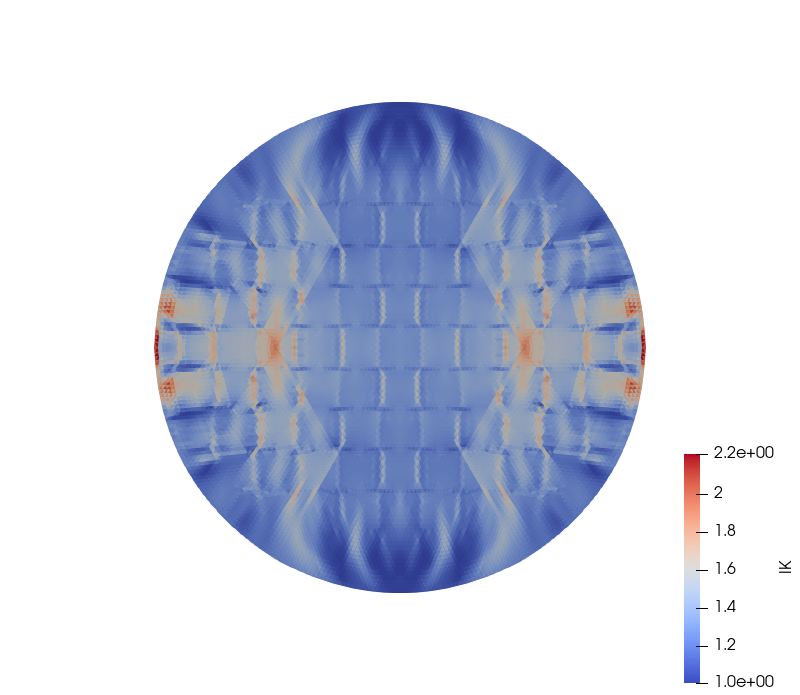}
	\includegraphics[width=.47\textwidth, trim = 5.3cm 0.2cm 0.3cm 3.5cm, clip]{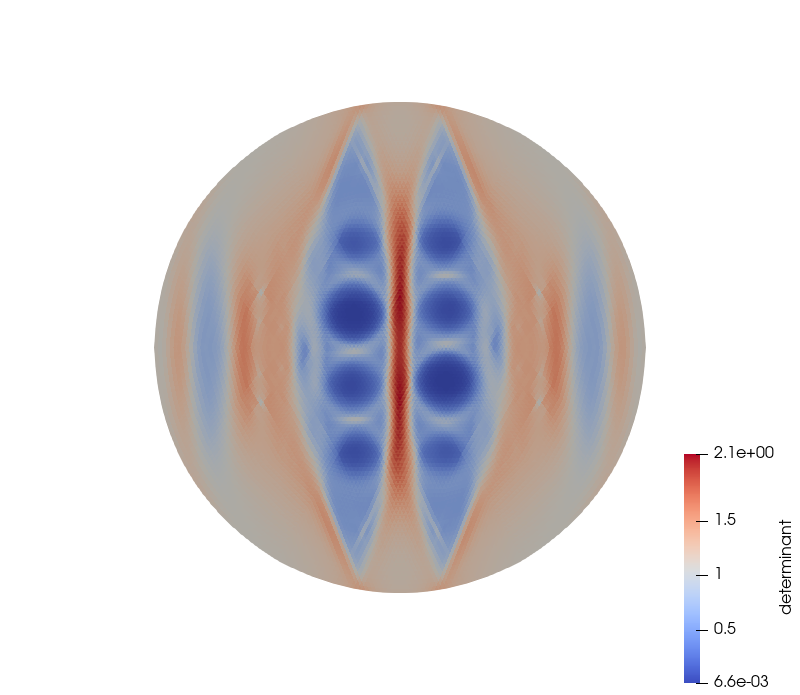}
	\hfill
	\includegraphics[width=.47\textwidth, trim = 5.3cm 0.2cm 0.3cm 3.5cm, clip]{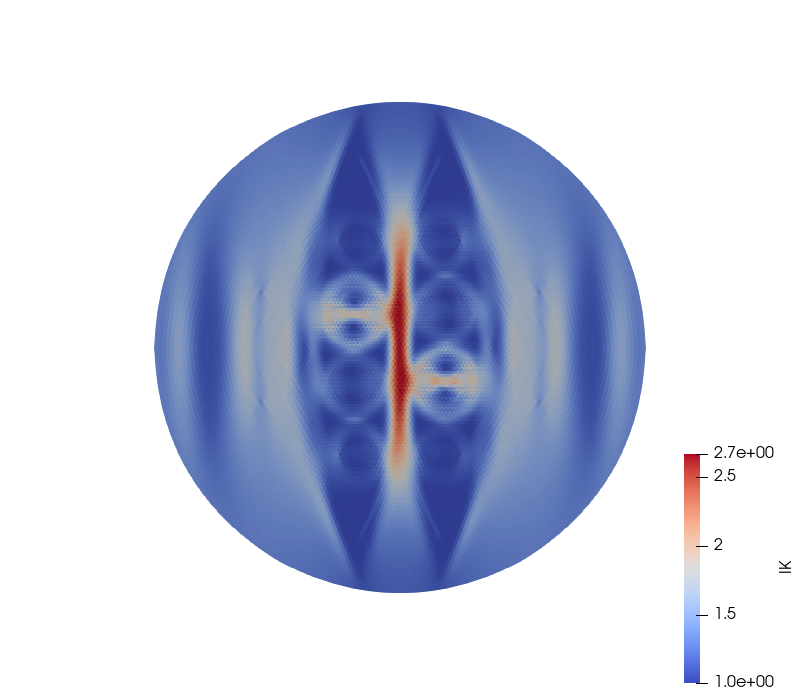}
	\caption{
		Two compatible deformations (for $L_c=0.5$ and $L_c=2$, respectively) starting from a checkerboard pattern for $F_0=\diag\bigl(\sqrt{2},\frac{1}{\sqrt{2}}\bigr)$. This pattern remains visible for the first calculation because as explained in the text, the trust-region algorithm did not find a local minimizer for $I_2(P)\to\min$.}
	\label{fig:Curl4}
\end{figure}
%
%
%
%
%
%%%%%%%%%%%%%%%%%%%%%%%%%%%%%%%%%%%%%%%%%%%%%%%%%%%%%%%%%%%%%%%%%
\section{Gradient Young measures and laminates}\label{sec:guerra}

More recently, the problem of Morrey's conjecture has been approached from the point of view of \emph{gradient Young measures} and \emph{laminates} \cite{kinderlehrer1991caracterisationCR,kinderlehrer1991characterizations,kinderlehrer1994gradient,
guerra2020numerical,guerra2019extremal,harris2018two}. In the following, we give a brief overview of this alternative approach and its relation to the optimization methods used above.

Throughout this section, let $\Omega=[0,1]^2$ denote the unit square.\footnote{Note that the term $\meas{\Omega}=1$ will often be omitted in the following, but would need to be taken into account carefully for a more general choice of $\Omega\subset\R^2$.} In order to positively answer Morrey's conjecture, we would need to find a rank-one convex function $W\col\R^{2\times2}\to\R$ that is not quasiconvex. More explicitly, this task can be stated as follows:
\begin{itemize}
	\item[(I)]
		Find a rank-one convex function $W\col\R^{2\times2}\to\R$, a Lipschitz mapping $\varphi\col\Omega\to\R^2$ and $F_0\in\R^{2\times2}$ with $\varphi(x)=F_0.x$ for all $x\in\partial\Omega$ such that
		\begin{equation}\label{eq:morreyClassical}
			\int_\Omega W(\grad\varphi(x))\,\dx < W(F_0)
						\,.
		\end{equation}
\end{itemize}
In particular, such a full solution to Morrey's problem requires both an energy function $W\col\R^{2\times2}\to\R$ and a mapping $\varphi\col\Omega\to\R^2$.

The classical approach to this problem, as followed in the previous sections, consists of choosing a plausible candidate $W$ first before trying to find a mapping $\varphi$ that satisfies \eqref{eq:morreyClassical} in order to prove the non-quasiconvexity of $W$. However, recent investigations of Morrey's conjecture have often taken a slightly different point of view, emphasizing the search for a suitable mapping $\varphi$ instead. More specifically, these approaches try to establish the existence of a Lipschitz mapping $\varphi\col\Omega\to\R^2$ such that the pushforward measure induced by $\grad\varphi$ violates a Jensen-type inequality.

This change in perspective is mainly based on the seminal results by Kinderlehrer and Pedregal \cite{kinderlehrer1991caracterisationCR,kinderlehrer1991characterizations,kinderlehrer1994gradient} on the relation between quasiconvexity and gradient Young measures. In particular, these results allow for Morrey's conjecture (I) to be rephrased in terms of properties of probability measures.
%
%
%
%%%%%%%%%%%%%%%%%%%%%%%%%%%%%%%%%%%%%%%%%%%%%%%%%%%%%%
\subsection{Equivalent formulations of Morrey's conjecture}

First, we consider the following rephrasing, which is obtained by a simple substitution on the left-hand side of \eqref{eq:morreyClassical}.
\begin{itemize}
	\item[(II)]
		Find a rank-one convex function $W\col\R^{2\times2}\to\R$, a Lipschitz mapping $\varphi\col\Omega\to\R^2$ and $F_0\in\R^{2\times2}$ with $\varphi(x)=F_0.x$ for all $x\in\partial\Omega$ such that
		\[
			\int_{\R^{2\times2}} W(A)\,\intd\nu_{\varphi}(A) < W(F_0)
			\,,
		\]
		where $\nu_{\varphi}$ denotes the pushforward measure of the Lebesgue measure with respect to $\grad\varphi$, i.e.\
		\[
			\nu_{\varphi}(M)=\meas{\{x\in\Omega\setvert \grad\varphi(x)\in M\}}\qquad\text{for any measurable}\quad M\subset\R^{2\times2}\,.
		\]
\end{itemize}
This phrasing of Morrey's problem in terms of \emph{probability measures}\footnote{Note that $\nu_{\varphi}$ is indeed a probability measure on $\R^{2\times2}$ due to the choice $1=\meas{\Omega}=\nu_{\varphi}(\R^{2\times2})$.} is closely related to its formulation in terms of \emph{homogeneous gradient Young measures}. The exact relation between pushforward measures of gradients and homogeneous gradient Young measures can be established by the \emph{Averaging Theorem} \cite[Theorem 2.1]{kinderlehrer1991characterizations},
which directly implies the following lemma as a corollary.
\begin{lemma}{{Cf.~\cite{kinderlehrer1991characterizations}}}
\label{lemma:averaging}
	Let $\varphi\in\sob1\infty(\Omega;\R^2)$ with $\varphi(x)=F_0.x$ for all $x\in\partial\Omega$. Then $\nu_{\varphi}$ is a homogeneous gradient Young measure on $\Omega$ with barycenter $\bary{\nu}_{\varphi}=F_0$.
\end{lemma}
By virtue of Lemma \ref{lemma:averaging}, if (II) can be solved, so can the following problem:
\begin{itemize}
	\item[(III)]
		Find a rank-one convex function $W\col\R^{2\times2}\to\R$ and a homogeneous gradient Young measure $\nu$ such that
		\begin{equation}\label{eq:morreyYoungMeasure}
			\int_{\R^{2\times2}} W(A)\,\intd\nu(A) < W(\bary{\nu})
			\,.
		\end{equation}
\end{itemize}
In order to see that (III) is also sufficient for (and thus equivalent to) solving (II), recall that for any homogeneous gradient Young measure $\nu$ there exists a sequence $(\varphi_k)_{k\in\N}\subset\sob1\infty(\Omega;\R^2)$ with affine linear boundary values induced by $F_0=\bary{\nu}$ such that
\[
	\lim_{k\to\infty}
	\int_{\R^{2\times2}} W(A)\,\intd\nu_{\varphi_k}(A)
	=
	\int_{\R^{2\times2}} W(A)\,\intd\nu(A)
\]
for any continuous function $W\col\R^{2\times2}\to\R$. In particular, this implies $\int_{\R^{2\times2}} W(A)\,\intd\nu_{\varphi_k}(A) < W(\bary{\nu})=W(F_0)$ for sufficiently large $k$ if $\nu$ satisfies \eqref{eq:morreyYoungMeasure}.

Finally, we can rephrase (III) by employing the notion of a \emph{laminate}. In the planar case, a laminate can be characterized as a probability measure $\nu$ on $\R^{2\times2}$ such that the Jensen-type inequality
\begin{equation}\label{eq:jensenTypeInequality}
	W(\bary{\nu}) \leq \int_{\R^{2\times2}} W(A)\,\dnu(A)
\end{equation}
holds for any rank-one convex function $W\col\R^{2\times2}\to\R$ \cite{pedregal1993laminates}. Since $\nu$ is a homogeneous gradient Young measure if and only if \eqref{eq:jensenTypeInequality} holds for any quasiconvex energy $W$ according to the Kinderlehrer-Pedregal Theorem \cite{kinderlehrer1991characterizations,kinderlehrer1994gradient}, the following problem is equivalent to (III):
\begin{itemize}
	\item[(IV)]
		Find a homogeneous gradient Young measure $\nu$ which is not a laminate.
\end{itemize}
Note that (IV) seems to make no reference to any energy function $W$. In practice, however, (approximately) rank-one convex energy functions need to be applied to a given measure $\nu$ in order to numerically establish whether it is a laminate \cite{guerra2020numerical}. On the other hand, it is often obvious by construction that $\nu$ is a homogeneous gradient Young measure. In particular, due to Lemma \ref{lemma:averaging}, this is the case if $\nu$ is obtained as the pushforward measure with respect to the gradient of a mapping $\varphi\col\Omega\to\R^2$.
%
%
%
%%%%%%%%%%%%%%%%%%%%%%%%%%%%%%%%%%%%%%%%%%%%%%%%%%%%%%%%%
\subsection{The numerical search for non-laminate gradient Young measures}

A specific numerical method for finding non-laminate homogeneous gradient Young measures numerically has been suggested by Guerra et al.\ \cite{guerra2020numerical}. This approach is based on selecting a dense subset $K(\Omega)\subset\sob1\infty(\Omega;\R^2)$ such that each $\varphi\in K(\Omega)$ induces a discrete-valued gradient field $\grad\varphi$. Then the following problem needs to be solved numerically:
\begin{itemize}
	\item[(V)]
		Find $\varphi\in K(\Omega)$ such that $\nu_{\varphi}$ with $\nu_{\varphi}(M)=\meas{\{x\in\Omega\setvert \grad\varphi(x)\in M\}}$ is not a laminate, where $K(\Omega)=\bigcup_{N\in\N}K_N$ with
		\begin{align}
			K_N&=\Bigl\{\varphi\in C(\Omega) \setvert \varphi(x)=\sum_{i=1}^N h(\iprod{x,\eta_i}+c_i)\.\xi_i\,,\; \eta_i,\xi_i\in\R^2\,,\; c_i\in\R\Bigr\}\,,\label{eq:guerraKN}\\
			h(t)&=t\.\chi_{[0,\afrac12]}(t) + (1-t)\.\chi_{[\afrac12,1]}(t)\notag
			\,.
		\end{align}
\end{itemize}
Both the barycenter $\bary{\nu}_\varphi$ and the energy value $\int_{\R^{2\times2}} W(A)\,\intd\nu(A)$ can be easily computed numerically for the measure $\nu_{\varphi}$ corresponding to any $\varphi\in K(\Omega)$ and a given energy function $W$ on $\R^{2\times2}$. However, in order to demonstrate that a given measure $\nu_{\varphi}$ is not a laminate, it is necessary to find a rank-one convex energy $W$ for which \eqref{eq:jensenTypeInequality} is violated. The numerical approach suggested by Guerra et al.\ \cite{guerra2020numerical} consists of generating a large number of (approximately) rank-one convex energy functions by computing the rank-one convex envelope for a specific class of functions based on earlier considerations by {\v{S}}ver{\'a}k \cite{sverak1992rank}.
%
%
%
%%%%%%%%%%%%%%%%%%%%%%%%%%%%%%%%%%%%%%%%%%%%%%%%%%%%%%%%%

\subsubsection{Application to $\Wmp$}
In contrast, based on our results in \cite{agn_voss2021morrey}, we conjecture that the functional $\Wmp$ defined in \ref{eq:WmpinK} is a good candidate energy to detect non-laminate measures. Therefore, we do not need to perform the computationally expensive calculations of multiple rank-one convex envelopes. Instead, we directly include the class $K(\Omega)$ of deformations in our search for a counterexample to the quasiconvexity inequality \eqref{eq:quasiconvexity}. Since the energy values of the inhomogeneous deformations $\varphi\in K(\Omega)$ and of the corresponding homogeneous deformations given as the energy at the barycenter $W(F_0)=W(\bary{\nu_{\varphi}})$ are easily computed numerically (without the need of FEM), we can thereby test $\Wmp$ with a large number of additional deformations with periodic boundary conditions. We call these inhomogeneous deformations $\varphi\in K(\Omega)$ \emph{combined laminates}.

\begin{figure}[h!]
	\centering
	\includegraphics[width=.49\textwidth,trim = 2.3cm 2.6cm 2.5cm 3cm, clip]{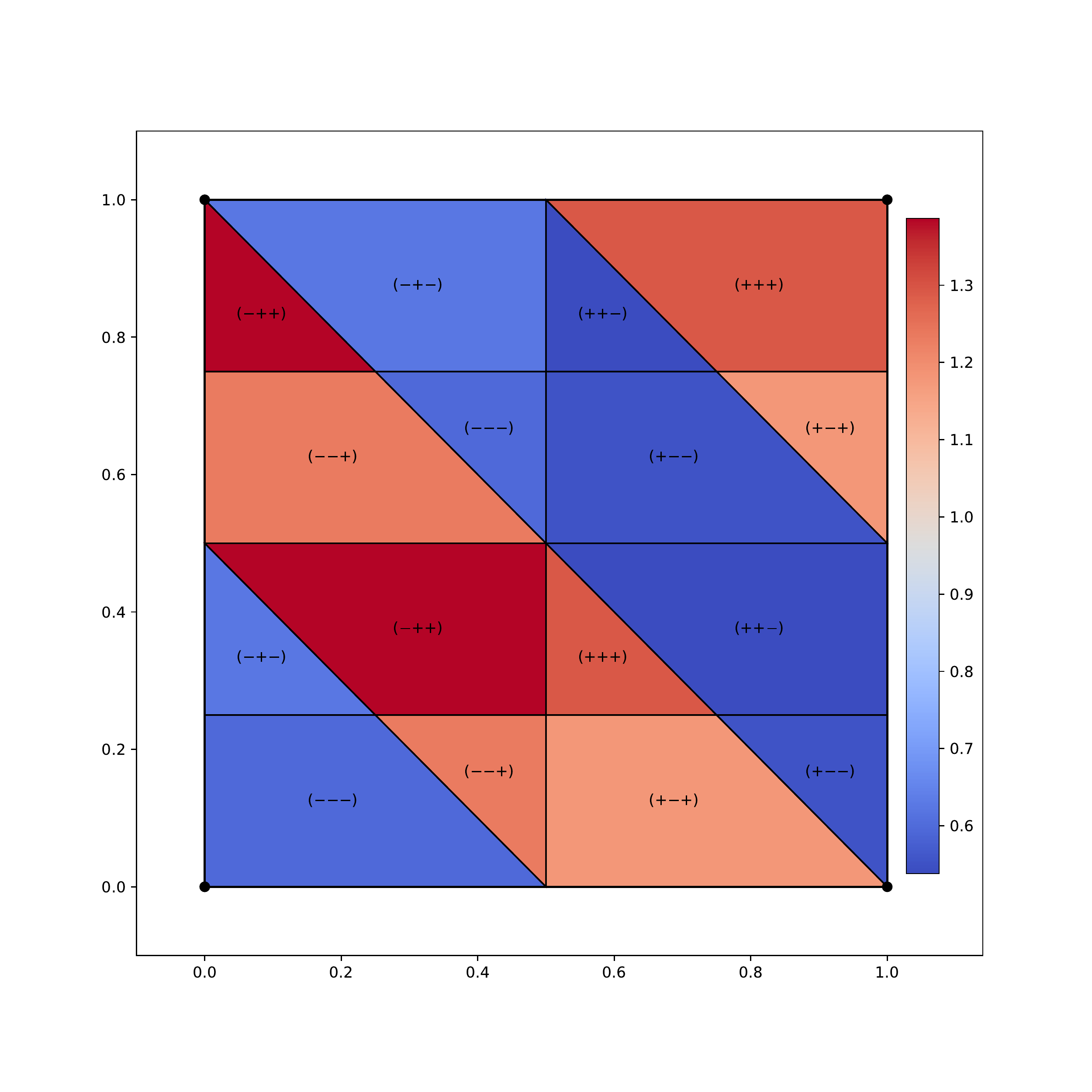}
  	\hfill
	\includegraphics[width=.49\textwidth,trim = 2.3cm 2.6cm 2.5cm 3cm, clip]{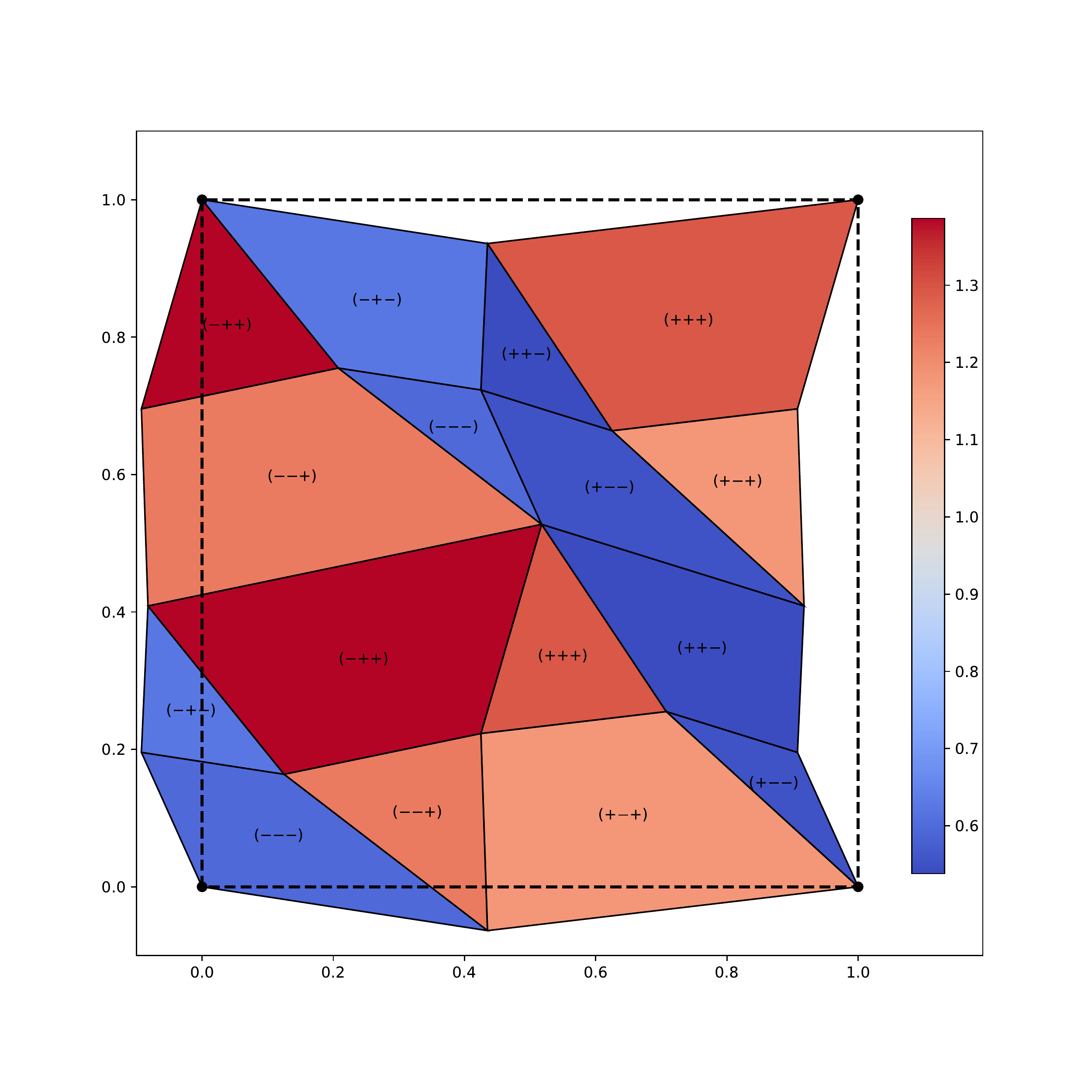}
	\caption{\label{fig:combinedLaminatesSimple}A possible combined laminate consisting of three parts: reference configuration (left) and deformed configuration (right). The deformation is homogeneous for each area with the coloring showing the corresponding value of $\det F$.}
\end{figure}

For a given homogeneous deformation gradient $F_0$ and given number $N$ of laminates to combine, cf.~\eqref{eq:guerraKN}, our implementation selects random parameter values for $\eta_i,\xi_i,c_i$ with $i\in\{1,\cdots,N\}$, validating that $\det F>0$ for all such combinations, and computes the energy value of the resulting combined laminate. Figures \ref{fig:combinedLaminatesSimple} and \ref{fig:combinedLaminates} show examples of such deformations with $N=3$ and $F_0=\id$; the \enquote{phases} of the superimposed deformations, i.e.\ the local values of $h'(\iprod{x,\eta_i}+c_i)$, are indicated by $(+/-)$.

\begin{figure}[h!]
	\centering
	\includegraphics[width=.49\textwidth,trim = 2.3cm 2.6cm 2.5cm 3cm, clip]{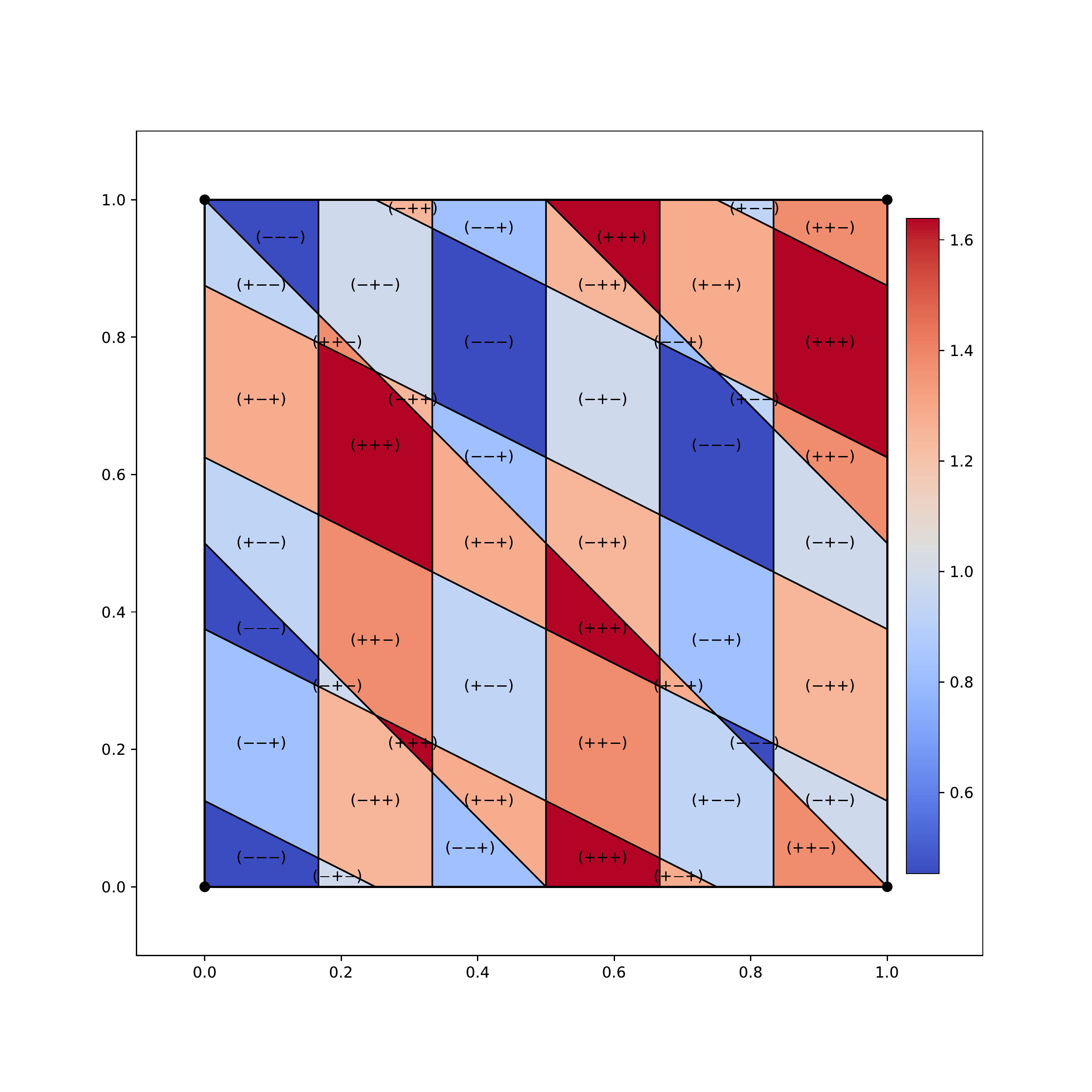}
  	\hfill
	\includegraphics[width=.49\textwidth,trim = 2.3cm 2.6cm 2.5cm 3cm, clip]{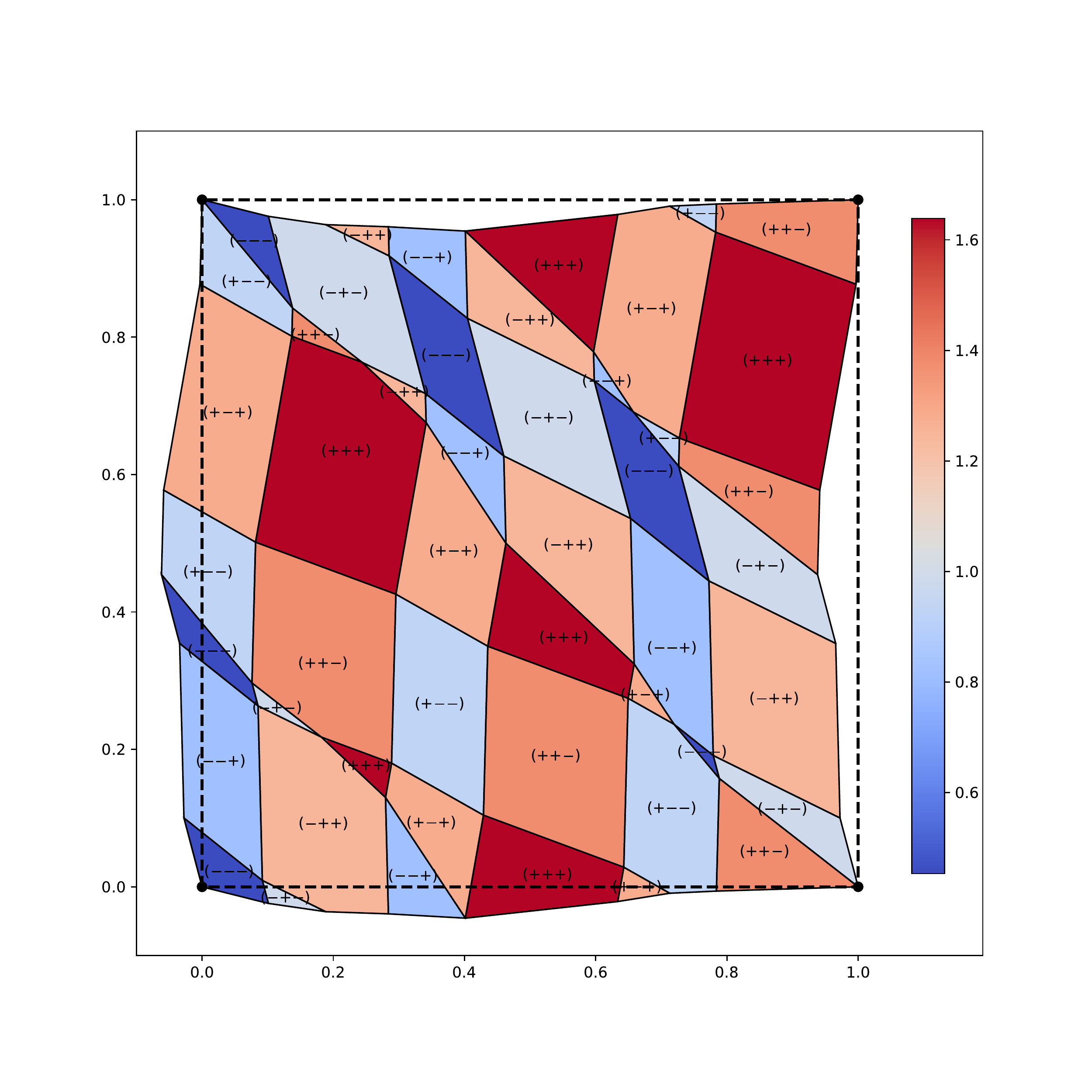}
	\caption{\label{fig:combinedLaminates}A possible combined laminate consisting of three parts of the reference configuration (left) and the deformed configuration (right). The deformation is homogeneous for each area with the coloring showing the corresponding value of $\det F$.}
\end{figure}

Our numerical tests focused on the case $N\geq5$.\footnote{Configurations with $N>3$ are usually too convoluted to properly visualize, especially with the addition $(+/-)$ to indicate the \enquote{phases}.} After testing slightly more than one million combinations with a routine written in Python, where we started with different $F_0=\diag(\sqrt{a},\frac{1}{\sqrt{a}})$ with $a\in[1,10]$ and tried $N=4,5,6,7$, we were once more unable to obtain an energy level below $\Wmp(F_0)$. As with our previous approaches, we always find non-trivial microstructures when we change our energy function to be non-rank-one convex. Figure \ref{fig:combinedLaminatesLower}, for example, shows such a microstructure for the energy $W_c$ defined in \eqref{eq:Wc} with a modified volumetric part $c\.\log\det F$ in place of $\Wmp$; for $c>1$, this method once more finds random configurations with lower energy than the homogeneous state.

\begin{figure}[h!]
	\centering
	\includegraphics[width=.49\textwidth,trim = 2.3cm 2.6cm 2.5cm 3cm, clip]{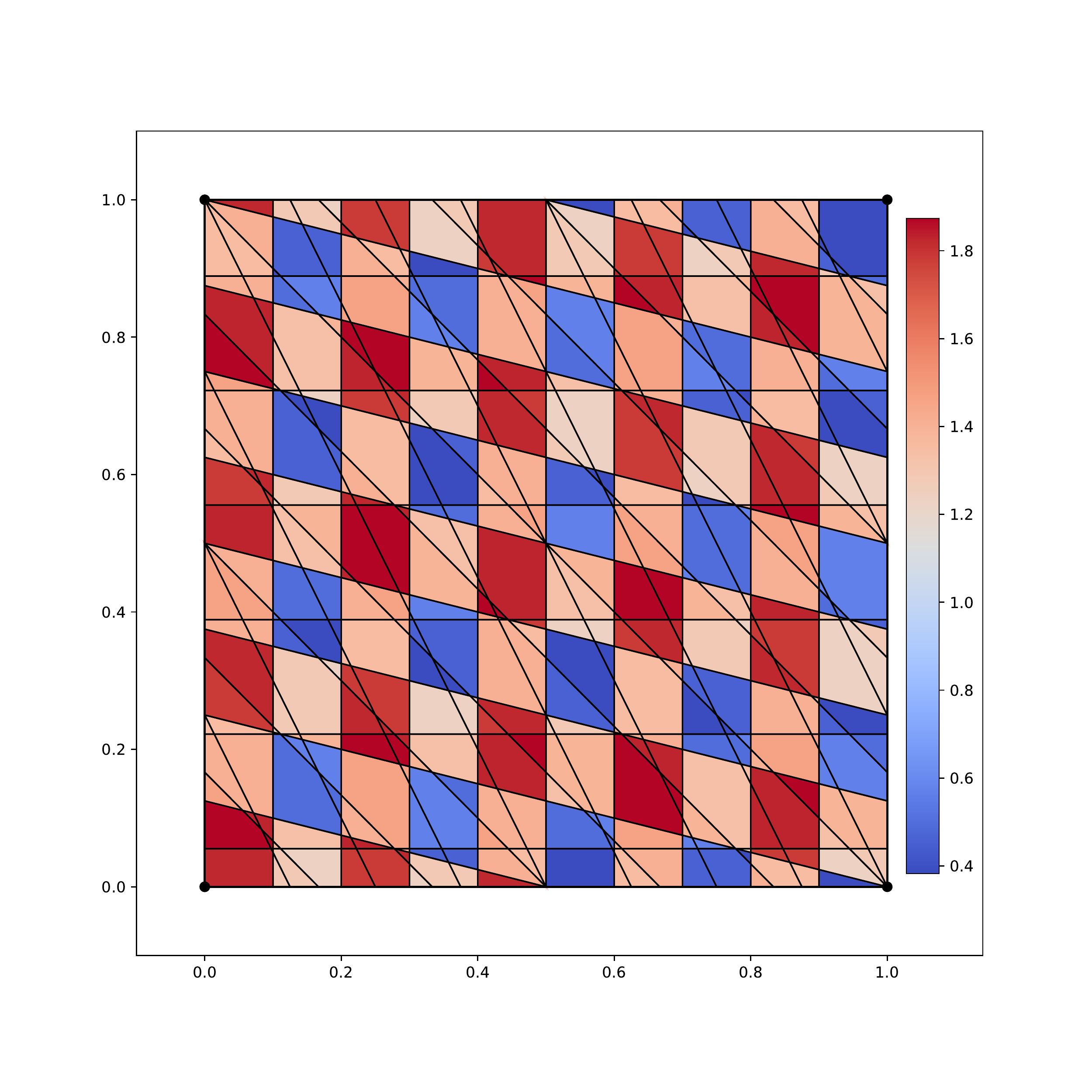}
	\hfill
	\includegraphics[width=.49\textwidth,trim = 2.3cm 2.6cm 2.5cm 3cm, clip]{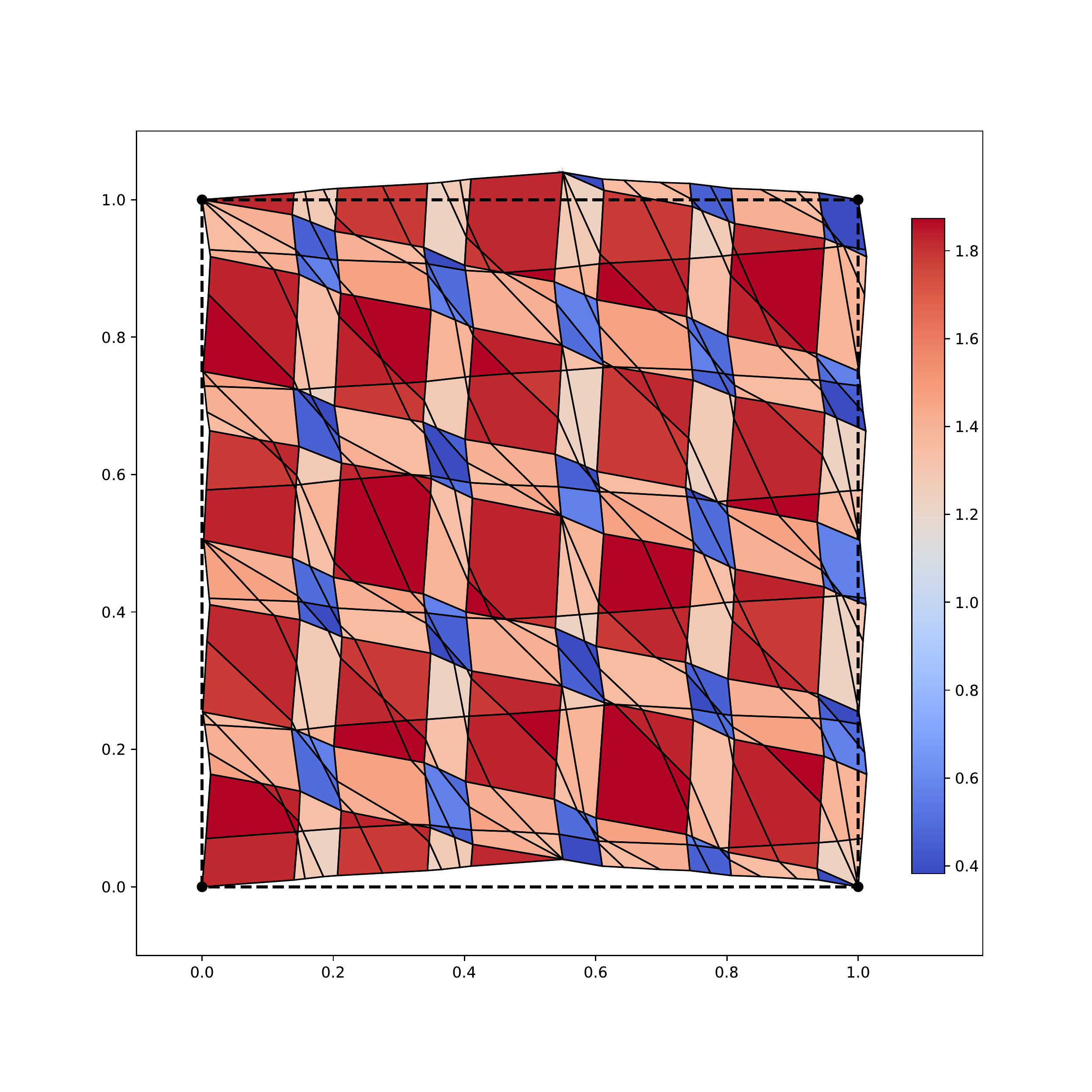}
	\caption{\label{fig:combinedLaminatesLower}A combined laminate consisting of five parts of the reference configuration (left) and the deformed configuration (right). The deformation is homogeneous for each area with the coloring showing the corresponding value of $\det F$. The energy value is lower than the homogeneous one for the non-elliptic energy density $W_c$ defined in \eqref{eq:Wc} with $c=1.5$.}
\end{figure}
%
%
%
%
%%%%%%%%%%%%%%%%%%%%%%%%%%%%%%%%%%%%%%%%%%%%%%%%%%%%%%%%%%%%%%%%%
\section{Discussion}

We have presented several different numerical approaches to check for quasiconvexity of a given function $W$.
\begin{itemize}
	\item In Section \ref{sec:numericOliver}, we demonstrated a classical finite element approach that can find easily microstructures if we perturb the energy candidate to be slightly non-rank-one convex. In addition, we showed a method of disturbing the homogeneous structure of the solution by modifying the energy values on a subdomain and computing the microstructure resulting from this material inhomogeneity.
	\item In Section \ref{sec:periodic}, under the assumption of periodic boundary conditions, we introduced a numerical scheme that is based on deep neural networks and thereby discovered a new non-trivial microstructure (smooth laminates) with the same energy value as the homogeneous deformation for the considered energy function, which we then investigated analytically.
	\item For the relaxation technique considered in Section \ref{sec:Curl}, we extended our numerical calculations from gradient fields $F$ to general matrix field $P$ by introducing a penalty term based on $\Curl P$. This approach resulted in various non-compatible fields with lower energy value than the energy of the homogeneous deformation, even for a quasiconvex energy candidate. The fields found this way were then used as new starting configurations for the finite elements approach. This, unfortunately, did never lead to energies below the one of the homogeneous state.
	\item In Section \ref{sec:guerra}, we discussed an alternative numerically straightforward way of checking for quasiconvexity connected to the theory of gradient Young measures rather than to optimization. Several rank-one laminations were combined so that the resulting deformation remained piece-wise homogeneous and their energy value were compared to the homogeneous deformation. Again, we only found lower values for a non-rank-one convex energy density.
\end{itemize}
We tested all these methods with the energy $\Wmp$ from \cite{agn_voss2021morrey}. If we change this energy to be non-elliptic and thus non-quasiconvex, the presented methods were all able to produce non-trivial microstructures, i.e.\ deformations that are neither homogeneous nor a simple first-order laminate. This demonstrates their viability as numerical tests for quasiconvexity. On the other hand, while non-quasiconvexity of an energy $W$ can conceivably be proven by numerical methods (since identifying a single suitable deformation would be sufficient), no amount of numerical testing can be considered an actual proof that a given function is quasiconvex. However, if all the described methods fail to yield a deformation energetically more optimal than the homogeneous one -- as was the case for the energy $\Wmp(F)=\frac{\lambdamax}{\lambdamin}+2\.\log\lambdamin$ considered here -- then this should be interpreted as a strong indication that the function is indeed quasiconvex and therefore not a viable candidate for answering Morrey's conjecture.
%
%
%
%
%
%%%%%%%%%%%%%%%%%%%%%%%%%%%%%%%%%%%%%%%%%%%%%%%%%%%%%%%%%%%%%%%%%
\section{References}
\footnotesize
\printbibliography[heading=none]
%
%
%
%
%
%%%%%%%%%%%%%%%%%%%%%%%%%%%%%%%%%%%%%%%%%%%%%%%%%%%%%%%%%%%%%%%%%
\small
\begin{appendix}
%
%
%
%
%
%%%%%%%%%%%%%%%%%%%%%%%%%%%%%%%%%%
\section{Contracting deformations}\label{sec:RadialEulerLagrange}

In a previous article \cite{agn_voss2021morrey}, we found a surprising connection of $\Wmp$ to the work of Burkholder and Iwaniec \cite{burkholder1988sharp,astala2012quasiconformal,guerra2021automatic} in the field of complex analysis with the so-called Burkholder function\footnote{The question whether the Burkholder function is quasiconvex or not is discussed extensively in the literature with numerical evidence in favor of the former \cite{baernstein2011some}. Furthermore, Guerra and Kristensen give a new interesting discussion about $B_p$ being extremal in the class of p-homogeneous functions \cite{guerra2021automatic}.}
\begin{align}
	B_p(F)=-\left[\frac{p}{2}\.\det F+\left(1-\frac{p}{2}\right)\opnorm{F}^2\right]\opnorm{F}^{p-2}\,,
\end{align}
where $\opnorm{F}\colonequals\sup_{\norm{\xi}=1}\norm{F\.\xi}_{\R^2}=\lmax$ denotes the operator norm (i.e.\ the largest singular value) of $F$. There, we showed the existence of a first non-trivial example for the energy $\Wmp(F)$ for which equality with the homogeneous solution holds. We considered the radial symmetric deformation, i.e.\ functions which can be described by a mapping $v\col[0,\infty)\to\R$ such that
\begin{equation}
	\varphi(x)=v(\norm x)\.\frac{x}{\norm x}\qquad\text{with}\quad v(0)=0\,,\label{eq:radialMappings}
\end{equation}
and its subclasses of expanding and contracting deformations
\begin{align}
	\mathcal{V}_R\colonequals&\left\{v\in C^1([0,R])\,\big|\,v(0)=0\,,\;v(R)=R\,,\; \frac{v(r)}{r}\geq v'(r)\geq 0\quad\forall\; r\in[0,R]\right\},\tag{expanding}\\
	\mathcal{V}_R\inv\colonequals&\left\{v\in C^1([0,R])\,\big|\,v(0)=0\,,\;v(R)=R\,,\; v'(r)\geq \frac{v(r)}{r}\geq 0\quad\forall\; r\in[0,R]\right\}.\tag{contracting}
\end{align}
In addition to the proof given in \cite{agn_voss2021morrey}, we can verify that these deformations are equilibrium states of $B_p$ and $\Wmp$ by directly computing the  Euler-Lagrange equations with restriction to radial deformations. It is important to note that the Euler-Lagrange equations given in this section are only a necessary but not a sufficient condition to check for stationary points because they assume the restriction to radial deformations as an a priori constraint. For the full system of Euler-Lagrange equations, we have to compute the divergence of the full first Piola-Kirchhoff tensor $S_1(F)=DW(F)$ \eqref{eq:WmpS1} and insert radial deformations as a posteriori constraint afterward which would go beyond the scope here but has been verified with the support of Mathematica.

We start with the Burkholder functional in the class of expanding deformations $v\in\mathcal V_R$
\begin{equation}
	\int_{B_R(0)}B_p(F)\,\dx=\pi\.\int_0^R \underbrace{p\.\frac{v'(r)\.v(r)^{p-1}}{r^{p-2}}+(2-p)\.\frac{v(r)^p}{r^{p-1}}}_{\equalscolon\mathcal F(v,v',r)}\,\dr\,,
\end{equation}
and compute the corresponding Euler-Lagrange equation
\begin{align}
	\ddr\mathcal F_{v'}=\mathcal F_v\qquad&\iff&\ddr\left[p\.\frac{v^{p-1}}{r^{p-2}}\right]&=p(p-1)\.\frac{v'\.v^{p-2}}{r^{p-2}}+(2-p)p\.\frac{v^{p-1}}{r^{p-1}}\\
	&\iff& (p-1)\.\frac{v'\.v^{p-2}}{r^{p-2}}-(p-2)\frac{v^{p-1}}{r^{p-1}}&=(p-1)\.\frac{v'\.v^{p-2}}{r^{p-2}}+(2-p)\.\frac{v^{p-1}}{r^{p-1}}\,,\notag
\end{align}
which is always true. Thus any $v\in\mathcal V_R$ is an equilibrium with respect to the restricted class of radial deformations. Now we do the same computation for the energy density $\Wmp$ in the opposite class of contracting deformations $\mathcal V_R\inv$
\begin{equation}
	\int_{B_R(0)}\Wmp(F)\,\dx=2\pi\int_0^R\underbrace{r^2\.\frac{v'(r)}{v(r)}+2r\left(\log v(r)-\log r\right)}_{\equalscolon\mathcal F(v,v',r)}\,\dr
\end{equation}
and compute the corresponding Euler-Lagrange equation
\begin{align}
	\phantom{.}\hspace{2cm}\ddr\mathcal F_{v'}=\mathcal F_v\qquad&\iff&\ddr\left[\frac{r^2}{v}\right]&=-r^2\.\frac{v'}{v^2}+\frac{2\.r}{v}\hspace{3cm}\phantom{.}\\
	&\iff&\frac{2\.r}{v}-r^2\.\frac{v'}{v^2}&=-r^2\.\frac{v'}{v^2}+\frac{2\.r}{v}\,,\notag
\end{align}
which is always true. Thus any $v\in\mathcal V_R\inv$ is an equilibrium with respect to the restricted class of radial deformations.
%
%
%
%%%%%%%%%%%%%%%%%%%%%%%%%
\section{Curl: Euler-Lagrange equations}\label{sec:appendixCurl}

In the three-dimensional case, we have
\begin{align}
	\curl\matr{v_1\\v_2\\v_3}=\nabla\times\matr{v_1\\v_2\\v_3}=\matr{v_{3,y}-v_{2,z}\\v_{1,z}-v_{3,x}\\v_{2,x}-v_{1,y}}\qquad\implies\qquad\curl\matr{v_1(x,y)\\v_2(x,y)\\0}=\matr{0\\0\\v_{2,x}-v_{1,y}}.
\end{align}
Thus in the planar case we may introduce the operator
\begin{equation}
	\curld\col\R^2\to\R\,,\qquad\curld(v_1,v_2)\colonequals v_{2,x}-v_{1,y}\,.\label{eq:Curl2D}
\end{equation} 
In addition, in three dimensions, taking $\curl$ row-wise, it holds
\begin{align*}
	\Curl\matr{P_{11}&P_{12}&P_{13}\\P_{21}&P_{22}&P_{23}\\P_{31}&P_{32}&P_{33}}=\matr{\curl(P_{11},P_{12},P_{13})^T\\\curl(P_{21},P_{22},P_{23})^T\\\curl(P_{31},P_{32},P_{33})^T}\qquad\implies\qquad\Curl\matr{P_{11}&P_{12}&0\\P_{21}&P_{22}&0\\0&0&0}=\matr{\curl(P_{11},P_{12},0)^T\\\curl(P_{21},P_{22},0)^T\\0\quad 0\quad 0}.
\end{align*}
Accordingly, in the planar case we may define the operator
\begin{align}
	\Curld\col\R^{2\times 2}\to\R^2\,,\qquad\Curld\matr{P_{11}&P_{12}\\P_{21}&P_{22}}\colonequals\matr{\curld(P_{11},P_{12})\\\curld(P_{21},P_{22})}=\matr{P_{12,x}-P_{11,y}\\P_{22,x}-P_{21,y}}.
\end{align}
In order to mimic the behavior of the curl in the three-dimensional case ($\Curl\Curl\col\R^{3\times 3}\to\R^{3\times 3}$) we define moreover the operator
\begin{align}
	\CCurl\col\R^{2\times 2}\to\R^{2\times 2}\,,\qquad\CCurl\matr{P_{11}&P_{12}\\P_{21}&P_{22}}\colonequals\matr{P_{12,xy}-P_{11,yy}&P_{11,xy}-P_{12,xx}\\P_{22,xy}-P_{21,yy}&P_{21,xy}-P_{22,xx}},
\end{align}
which is the relevant $2\times 2$ block of the three-dimensional expression
\begin{align}
	\Curl\Curl\matr{P_{11}&P_{12}&0\\P_{21}&P_{22}&0\\0&0&0}=\Curl\matr{0&0&P_{12,x}-P_{11,y}\\0&0&P_{22,x}-P_{221,y}\\0&0&0}=\matr{P_{12,xy}-P_{11,yy}&P_{11,xy}-P_{12,xx}&0\\P_{22,xy}-P_{21,yy}&P_{21,xy}-P_{22,xx}&0\\0&0&0}\notag
\end{align}
and ensures the common condition $\Div\CCurl(P)=0$. We calculate the weak form of the Euler-Lagrange equations
\begin{align}
	\left.\ddt I_1(\varphi+t\.\vartheta)\right|_{t=0}&=\int_\Omega\left.\ddt \Wmp(\nabla\varphi+t\.\nabla\vartheta)\right|_{t=0}\dx=\int_\Omega\iprod{D\Wmp(\nabla\varphi),\nabla\vartheta}\,\dx\\
	&\overset{\mathclap{P.I.}}{=}\;\int_\Omega -\iprod{\Div D\Wmp(\nabla\varphi),\vartheta}\,\dx=0\qquad\text{for all }\vartheta\in C_0^\infty(\Omega)\,,\notag\\
	\left.\ddt I_2(P+t\.\delta P)\right|_{t=0}&=\int_\Omega\left[\ddt \Wmp(P+t\.\delta P)+\frac{1}{2}L_c^2\.\norm{\Curld(P+t\.\delta P)}^2\right]_{t=0}\dx\notag\\
	&=\int_\Omega\iprod{D\Wmp(P),\delta P}+L_c^2\iprod{\Curld P,\Curld\delta P}\,\dx\label{eq:curlI2variation}\\
	&\overset{\mathclap{P.I.}}{=}\int_\Omega\iprod{D\Wmp(P)+L_c^2\.\CCurl P,\delta P}\,\dx=0\qquad\text{for all }\delta P\in C_0^\infty(\R^2;\R^{2\times 2})\,,\notag
\end{align}
where we used that $\Curl$ is a linear operator. With the fundamental lemma of the calculus of variations we obtain the strong form of the Euler-Lagrange equations for the respective cases
\begin{align}
	\Div D\Wmp(\nabla\varphi)&=0\qquad\text{(compatible)}\,,\\
	D\Wmp(P)+L_c^2\.\CCurl P&=0\qquad\text{(incompatible)}\,.\label{eq:curlStrongForm}
\end{align}
Because of $\Div\CCurl(P)=0$ the Euler-Lagrange equation \eqref{eq:curlStrongForm} always implies
\begin{equation}
	\Div D\Wmp(P)=0\,,
\end{equation}
where
\[
	\Div\matr{P_{11}&\cdots&P_{1n}\\\vdots&&\vdots\\P_{n1}&\cdots&P_{nn}}=\matr{\div(P_{11},\cdots,P_{1n})\\\vdots\\\div(P_{n1},\cdots,P_{nn})}
\]
denotes the row-wise divergence of $P\in C^1(\Rn;\Rnn)$.

Since in the planar case, $\curl$ can be seen as a rotated $\div$, i.e.
\begin{equation}
	\curld(v_1,v_2)=\div(v_2,-v_1)=\div\left[(v_1,v_2)\.\matr{0&-1\\1&0}\right]=\div\left[\matr{0&1\\-1&0}\matr{v_1\\v_2}\right]
\end{equation}
we have\footnote{Note that we use $\div(v_1,\cdots,v_n)=\div\matr{v_1\\\vdots\\v_n}$ and $\curl(v_1,\cdots,v_n)=\curl\matr{v_1\\\vdots\\v_n}$ for both row and column vectors.}
\begin{align}
	\Curld\matr{P_{11}&P_{12}\\P_{21}&P_{22}}=\matr{P_{12,x}-P_{11,y}\\P_{22,x}-P_{21,y}}&=\Div\matr{P_{12}&-P_{11}\\P_{22}&-P_{21}}=\Div\left[\matr{P_{11}&P_{12}\\P_{21}&P_{22}}\matr{0&-1\\1&0}\right]\notag\\
	\iff\qquad\Curld P&=\Div[P\.Q]\qquad\text{with }Q=\matr{0&-1\\1&0}.
\end{align}
Thus looking\footnote{You arrive at the same expression \eqref{eq:CCurlDiv} by computing $\left(\nabla\Div[P\.Q]\right)Q^T$ directly.} at equation \eqref{eq:curlI2variation}
\begin{align}
	\int_\Omega\iprod{\Curl P,\Curl\delta P}\,\dx&=\int_\Omega\iprod{\Div[P\.Q],\Div[\delta P\.Q]}\,\dx\;\overset{\mathclap{P.I.}}{=}\;-\int_\Omega\iprod{\nabla\Div[P\.Q],\delta P\.Q}\,\dx\notag\\
	\implies\qquad \CCurl P&=-\left(\nabla\Div[P\.Q]\right)Q^T\,,\qquad\text{with}\qquad Q=\matr{0&-1\\1&0}.\label{eq:CCurlDiv}
\end{align}
Therefore,
\begin{equation*}
	\left.\ddt I_2(P+t\.\delta P)\right|_{t=0}=\int_\Omega\iprod{D\Wmp(P)-L_c^2\left(\nabla\Div[P\.Q]\right)Q^T,\delta P}\,\dx\,,
\end{equation*}
which corresponds to the Euler-Lagrange equation in strong $\nabla\Div$-form
\begin{equation}
	D\Wmp(P)-L_c^2\left(\nabla\Div[P\.Q]\right)Q^T=0
\end{equation}
which is equivalent to \eqref{eq:curlStrongForm}. The corresponding minimization problem \eqref{eq:curlI2} can hence be rewritten using $\Curld P=\Div[P\.Q]$ as
\begin{equation}
	I_2(P)=\int_\Omega\Wmp(P)+\frac{L_c^2}{2}\.\norm{\Div[P\.Q]}^2\,\dx\to\min.\qquad P.\tau|_{\partial\Omega}=F_0.\tau\,,
\end{equation}
We may now substitute $\Phat\colonequals P\.Q\iff P=\Phat\.Q^T$. Then it holds
\begin{align}
	D\Wmp(\Phat\.Q^T)&=S_1(\Phat\.Q^T)=\sigma(\Phat\.Q^T)\Cof(\Phat\.Q^T)\overset{\star}{=}\sigma(\Phat)\Cof(\Phat\.Q^T)\notag\\
	&=\sigma(\Phat)\.\Cof(\Phat)\.Q^T=S_1(\Phat)\.Q^T=D\Wmp(\Phat)\.Q^T,
\end{align}
because of the isotropy $(\star)$ of $\Wmp$. Altogether,
\begin{align}
	\left.\ddt I_2(P+t\.\delta P)\right|_{t=0}&=\int_\Omega\iprod{D\Wmp(\Phat)\.Q^T-L_c^2\left(\nabla\Div\Phat\right)Q^T,\delta P}\,\dx\notag\\
	&=\int_\Omega\iprod{D\Wmp(\Phat)-L_c^2\.\nabla\Div\Phat,(\delta P)\.Q}\,\dx
\end{align}
which now corresponds to the Euler-Lagrange equation in strong $\nabla\Div$-form for $\Phat$
\begin{equation}
	D\Wmp(\Phat)-L_c^2\.\nabla\Div\Phat=0
\end{equation}
and the new minimization problem
\begin{equation}
	I_2^*(\Phat)\colonequals\int_\Omega\Wmp(\Phat)+\frac{L_c^2}{2}\.\norm{\Div\Phat}^2\,\dx\to\min\,,\qquad (\Phat\.Q^T).\tau|_{\partial\Omega}=F_0.\tau\,,
\end{equation}
with $\tau\in\R^2$ as the unit tangent vector to $\partial\Omega$. With $\tau=Q.\nu$ (since $\iprod{Q.\tau,\tau}=0$ in the planar case) we compute
\begin{equation*}
	(\Phat\.Q^T).\tau=F_0.\tau\qquad\iff\qquad(\Phat\.Q^T).(Q.\nu)=F_0.(Q.\nu)\qquad\iff\qquad\Phat.\nu=(F_0\.Q).\nu
\end{equation*}
and transform the tangential boundary conditions in the curl-formulation to the Neumann-boundary conditions for $\Phat$ in the div-formulation
\begin{equation}
	I_2^*(\Phat)\colonequals\int_\Omega\Wmp(\Phat)+\frac{L_c^2}{2}\.\norm{\Div\Phat}^2\,\dx\to\min\,,\qquad \Phat.\nu|_{\partial\Omega}=(F_0\.Q).\nu\,,
\end{equation}
with $\nu\in\R^2$ as the unit normal vector to $\partial\Omega$. 

For numerical experiments we may then also use the space
\[
	H(\Div;\Omega)=\left\{P\in L^2(\Omega)\;|\;\Div P\in L^2(\Omega)\right\}.
\]
Because of the involved nonlinear form of $\Wmp(P)$ it is, however, not clear whether or not $P\in L^2(\Omega)$ holds.\footnote{For the modified energy
\begin{equation}
	W_0(F)=\Wmp(F)+\frac{1}{\det F}=\K(F)+\sqrt{\K(F)^2-1}-\arcosh\K(F)+\log\det F+\frac{1}{\det F}\,,
\end{equation}
it is possible to achieve $W_0(F)\geq c^+\norm F^q$.}
\end{appendix}
\end{document}